\documentclass[12pt,a4paper]{amsart}

\usepackage[pagebackref]{hyperref}
\hypersetup{colorlinks=true,urlcolor=blue,citecolor=blue,linkcolor=blue}
\usepackage[nameinlink]{cleveref}
\usepackage{fullpage}

\usepackage{colonequals}
\usepackage{bbm}

\usepackage{amssymb,amsxtra,stmaryrd}

\newenvironment{enumalph}
{\begin{enumerate}}
{\end{enumerate}}

\newenvironment{enumroman}
{\begin{enumerate}}
{\end{enumerate}}

\newenvironment{enumalg}
{\begin{enumerate}}
{\end{enumerate}}

\newenvironment{enumalgalph}
{\begin{enumerate}}
{\end{enumerate}}

\usepackage[all]{xy}

\newcommand{\psmod}[1]{~(\textup{\text{mod}}~{#1})}

\def\fraka{\mathfrak{a}}
\def\frakb{\mathfrak{b}}
\def\frakc{\mathfrak{c}}
\def\frakp{\mathfrak{p}}
\def\frakq{\mathfrak{q}}
\def\frakn{\mathfrak{n}}

\def\frakN{\mathfrak{N}}
\def\frakM{\mathfrak{M}}

\def\frakD{\mathfrak{D}}
\def\sM{M}

\def\hatx#1{\widehat{#1}^\times}

\def\set#1{\left\{{\def\st{\;:\;}#1}\right\}}
\def\<#1>{\left\langle{#1}\right\rangle}

\newcommand{\defi}[1]{{{\fontfamily{lmss}\selectfont \textit{#1}}}}
\newcommand{\abs}[1]{\lvert {#1} \rvert}

\newcommand{\C}{\mathbb{C}}

\newcommand{\F}{\mathbb{F}}

\newcommand{\HH}{\mathbb{H}}

\newcommand{\Q}{\mathbb{Q}}

\newcommand{\Z}{\mathbb{Z}}

\newcommand{\alphahat}{\widehat{\alpha}}
\newcommand{\betahat}{\widehat{\beta}}
\newcommand{\pihat}{\widehat{\pi}}
\newcommand{\Fhat}{\widehat{F}}
\newcommand{\Ghat}{\widehat{G}}
\newcommand{\Khat}{\widehat{K}}

\newcommand{\uhat}{\widehat{u}}

\newcommand{\calO}{O}  % quaternion order
\newcommand{\calOhat}{\widehat{\calO}}
\newcommand{\Bhat}{\widehat{B}}

\newcommand{\Snew}{S^{\textup{new}}}  % quaternion order

\newcommand{\tbigwedge}{\smash{\raisebox{0.2ex}{\ensuremath{\textstyle{\bigwedge}}}}}

\DeclareMathOperator{\adj}{adj}
\DeclareMathOperator{\AL}{AL}
\DeclareMathOperator{\Aut}{Aut}
\DeclareMathOperator{\Cl}{Cl}
\DeclareMathOperator{\Clf}{Clf}
\DeclareMathOperator{\Cls}{Cls}

\DeclareMathOperator{\opchar}{char}
\DeclareMathOperator{\End}{End}
\DeclareMathOperator{\disc}{disc}
\DeclareMathOperator{\discrd}{discrd}
\DeclareMathOperator{\id}{id}

\DeclareMathOperator{\Frac}{Frac}

\DeclareMathOperator{\Gen}{Gen}
\DeclareMathOperator{\GO}{GO}

\DeclareMathOperator{\Map}{Map}
\DeclareMathOperator{\GL}{GL}

\DeclareMathOperator{\Pic}{Pic}
\DeclareMathOperator{\diag}{diag}
\DeclareMathOperator{\SO}{SO}
\DeclareMathOperator{\OO}{O}
\DeclareMathOperator{\Orth}{O}

\DeclareMathOperator{\M}{M}
\DeclareMathOperator{\Ten}{Ten}

\DeclareMathOperator{\nrd}{nrd}

\DeclareMathOperator{\ord}{ord}

\DeclareMathOperator{\rad}{rad}

\DeclareMathOperator{\SL}{SL}

\DeclareMathOperator{\Sym}{Sym}
\DeclareMathOperator{\tr}{tr}

\DeclareMathOperator{\Typ}{Typ}

\DeclareMathOperator{\Rad}{rad}

\newcommand{\bdG}{\mathbf G}

\DeclareMathOperator{\Nm}{Nm}

\newcommand{\quat}[2]{\displaystyle{\biggl(\frac{#1}{#2}\biggr)}}

\numberwithin{equation}{section}
\newtheorem{theorem}[equation]{Theorem}
\newtheorem{thm}[equation]{Theorem}
\newtheorem{alg}[equation]{Algorithm}
\newtheorem{lemma}[equation]{Lemma}
\newtheorem{lem}[equation]{Lemma}
\newtheorem{proposition}[equation]{Proposition}
\newtheorem{prop}[equation]{Proposition}
\newtheorem{cor}[equation]{Corollary}

\newtheorem*{theorem*}{Theorem}

\theoremstyle{definition}
\newtheorem{defn}[equation]{Definition}

\theoremstyle{remark}
\newtheorem{remark}[equation]{Remark}
\newtheorem{rmk}[equation]{Remark}
\newtheorem{exm}[equation]{Example}

\usepackage{color}

\newcommand{\eps}{\varepsilon}

\newcommand{\Lambdahat}{\widehat{\Lambda}}

\newcommand{\tprodprime}[1]{\textstyle{\prod^{\prime}_{#1}}}

\title[Computing ternary orthogonal modular forms]{Computing Hilbert modular forms \\ as orthogonal modular forms}
\author{Jeffery Hein}
\address{}
\email{}

\author{Gonzalo Tornar\'ia} 
\address{Universidad de la Rep\'ublica, Montevideo, Uruguay}
\email{tornaria@cmat.edu.uy}

\author{John Voight}
\address{Department of Mathematics,
  Dartmouth College, 6188 Kemeny Hall, Hanover, NH 03755, USA}
\email{jvoight@gmail.com}

\begin{document}

\begin{abstract}
    We show how to efficiently compute Hilbert modular forms as orthogonal
    modular forms, generalizing and expanding upon the method of Birch.  
\end{abstract}

\maketitle

\setcounter{tocdepth}{1} 
\tableofcontents

%%%%%%%%%%%%%%%%%%%%%%%%%%%%%%%%%%%%%%%%%%%%%%%%%%%%%%%%%%%%%%%%%%%%%%%%

\section{Introduction}

\subsection*{Motivation}

Algorithms for the efficient computation of classical modular forms have applications in many areas in mathematics, and consequently their study and implementation remains a topic of enduring interest \cite{11authors}.  In 1991, Birch \cite{Birch} provided such an algorithm based on the Hecke action on classes of ternary quadratic forms.  Of his method, Birch says \cite[p.~204, p.~191]{Birch}: 
\begin{quote}
[T]here is a great deal of interesting information to be calculated; since the program is very fast, it is possible for anyone who owns it to generate interesting numbers much faster than it is possible to read them.  ...  

[However,]\ this attempt ...\ has so far failed in two ways: first, it usually gives only half the information needed, and, second, when the level is not square-free it gives even less information.  At least the program is very fast!
\end{quote}
Birch's algorithm computes the \emph{Anzahlmatrizen} of Eichler \cite[\S 17]{Eichler}, which were studied in their relation to Brandt matrices by Ponomarev \cite{Ponomarev} and Schulze-Pillot \cite[\S 2, Lemma 1]{Schulze-Pillot}. 

Birch's approach was first extended by Tornar\'ia \cite{Tornaria-thesis} in his PhD thesis,
where the notion of equivalence of quadratic lattices is refined to the notion of $\Theta$-equivalence. In this way, the Hecke module is
enlarged to obtain the missing information in the squarefree
case, although there is no statement or proof of what
eigenforms are constructed in general.  In his PhD thesis, Hein \cite{Hein-thesis} generalized Birch's construction to work over a totally real field and interpreted the modular forms obtained using the construction of the even Clifford algebra.

In this paper, we combine these approaches and exhibit an algorithm that simultaneously addresses both of Birch's issues, generalizes to Hilbert modular forms, and is still very fast: it gives as output all Hilbert modular forms of even weight and trivial character \emph{except} in the case where the totally real base field has odd degree and the level is a square.  

\subsection*{Results}

Most of the effort in this paper is theoretical, but we begin with the motivating and recognizable algorithmic application.  We consider the space of Hilbert cuspforms $S_k(\widehat{\Gamma}_0(\frakN))$ with trivial central character, specified by:
\begin{itemize}
\item a \defi{base field}, a totally real field $F$ with ring of integers $R \colonequals \Z_F$;
\item an even \defi{weight}, a vector $k=(k_v)_{v\mid\infty} \in (2\Z_{\geq 1})^{[F:\Q]}$; and
\item a \defi{level}, a nonzero ideal $\frakN \subseteq R$.
\end{itemize}
We abbreviate $\abs{k} \colonequals \sum_{v} k_v$.  For full definitions, see \cref{sec:hmf}.  
% (Because we are working with orthogonal groups, for representation-theoretic reasons we can only see forms with trivial central character and even weight.)
% For a weight $k$, let $\abs{k} \colonequals \sum_{v\mid\infty} k_v$.  

The space $S_k(\widehat{\Gamma}_0(\frakN))$ comes equipped with an action by Hecke operators $T_\frakn$ indexed by nonzero ideals $\frakn \subseteq \Z_F$ with $\frakn$ coprime to $\frakN$.  The Hecke operators act semisimply and pairwise commute.  %  $T_\frakm T_\frakn = T_\frakn T_\frakm$ whenever $\frakm,\frakn$ are coprime.  
The space $S_k(\widehat{\Gamma}_0(\frakN))$ further comes equipped with degeneracy operators $S_k(\widehat{\Gamma}_0(\frakN/\frakp)) \rightrightarrows S_k(\widehat{\Gamma}_0(\frakN))$ for all primes $\frakp \mid \frakN$, commuting with the Hecke operators.  Let $\Snew_k(\widehat{\Gamma}_0(\frakN))$ be the orthogonal complement of the images of these operators under the Petersson inner product, inheriting a Hecke action.
%\gt{For the proof I think we'll need $W_\frakq$ acting on all forms, not just new; but I guess not in the statement so maybe I'll mention that $W_\frakq$ acts on all forms when we need it and avoid opening the can of worms here.}
The space $\Snew_k(\widehat{\Gamma}_0(\frakN))$ is furthermore equipped with Atkin--Lehner involutions $W_{\frakq}$ for each prime power divisor $\frakq=\frakp^e \parallel \frakN$, again commuting with the Hecke operators by the theory of newforms.  For a \defi{sign vector} $\eps \in \prod_{\frakp \mid \frakN} \{\pm 1\}$, let
\begin{equation} 
\Snew_k(\widehat{\Gamma}_0(\frakN))^{\eps} \colonequals \{f \in \Snew_k(\widehat{\Gamma}_0(\frakN)) : W_\frakq f = \eps_\frakp f\} 
\end{equation}
be the subspace of forms with Atkin--Lehner eigenvalues matching the signs in $\eps$.  Then
\begin{equation} 
\Snew_k(\widehat{\Gamma}_0(\frakN)) = \bigoplus_{\eps} \Snew_k(\widehat{\Gamma}_0(\frakN))^{\eps}. 
\end{equation}

% \gt{We need to change $S_k(\frakN)$ to $S_k^{\text{new}}(\frakN)$, for instance, if $\frakN = p^2\,q$, $S_k(p^2q)$ includes newforms of level $p^2$ which can't always be constructed by OMF. Also, we need to restrict to \emph{trivial central character}, iow, forms fixed by the action of the class group of $F$.}

Our main algorithmic result is as follows.  

\begin{theorem} \label{thm:mainthm}
There exists an explicit algorithm that, given as input 
\begin{center}
a base field $F$, an even weight $k$, and a factored level $\frakN=\prod_{i=1}^r \frakp_i^{e_i}$ \\
such that $[F:\Q]$ is even or $\frakN$ is nonsquare, \\ 
and a sign vector $\eps$, 
\end{center}
computes as output the new space $\Snew_k(\widehat{\Gamma}_0(\frakN))^\eps$ 
% of Hilbert cusp forms over $F$ of weight $k$ and level $\frakN$ with Atkin--Lehner eigenvalues $\eps$, 
as a Hecke module.  

If $F$ and $\frakN$ are fixed, then this algorithm takes $\widetilde{O}(d^2 \Nm(\frakp))$ bit operations to compute the Hecke operator $T_\frakp$, where $d \colonequals \dim_\C \Snew_k(\widehat{\Gamma}_0(\frakN))^\eps$.  
\end{theorem}

To compute $\Snew_k(\widehat{\Gamma}_0(\frakN))^\eps$ \emph{as a Hecke module}, we mean to return for each ideal $\frakn$ coprime to $\frakN$ a matrix $[T_\frakn]$ representing the action of the Hecke operator $T_\frakn$ (with respect to a consistent choice of basis for the space).  

At least for $F=\Q$ \cite{Martin-sign} (and likely also true more generally), the Atkin--Lehner operators cut up the space of newforms into subspaces of approximately the same size, giving the expectation $\dim_\C \Snew_k(\widehat{\Gamma}_0(\frakN))^\eps \asymp 2^{-r} \Nm(\frakN)^{\abs{k-1}}$ as $\Nm(\frakN) \to \infty$ by the mass formula.  After a pre\-computation (to set up the lattice), the running time becomes $\widetilde{O}(d\Nm(\frakp))$, so approximately linear in the dimension.  To compute the systems of Hecke eigenvalues on $\Snew_k(\widehat{\Gamma}_0(\frakN))^\eps$, we could continue with techniques from linear algebra---our setup is especially well suited for this, as the matrices we produce are sparse.  
% this step is necessary for all methods to compute modular forms, and so we do not analyze it further here.  
% \gt{otoh, not all methods give matrices as sparse as these ones; i think the linear algebra here is as easy as possible, only matched by the quaternionic modular forms methods; in particular linear algebra for modular symbols seems to be "more difficult".}

The algorithm provided in \Cref{thm:mainthm} uses orthogonal modular forms on a ternary quadratic space and passing through quaternionic modular forms.  We succeed in computing the forms that Birch did not find by adding a character: for each sign vector $\varepsilon$, we define a \emph{radical character} $\nu_{\varepsilon}$ in \eqref{eqn:radicalchar}, 
% arising from the spinor norm \gt{not quite, think a bit about what to say here}, 
and we consider orthogonal modular forms that transform according to this character. 
% \gt{check this phrasing is ok, or should we say "orthogonal modular forms with nebentypus" ? }
% \jv{Nebentypus makes me think Dirichlet character, which is really not what we mean.}

To explain this in more detail, let $V$ be a totally positive definite ternary quadratic space over $F$.  Attached to $V$ is its even Clifford algebra $B \colonequals \Clf^0(V)$; let $\mathfrak{D}$ be its discriminant, the product of the ramified primes.  Let $\Lambda \subset V$ be an integral $R$-lattice such that $\Lambda^\sharp/\Lambda$ is cyclic.
% N is not needed here
% and let $\frakN \colonequals \disc \Lambda$ be its (half-)discriminant.  
%Factor $\frakN = \frakN_{\rm ram} \frakN_{\rm spl}$ where 
%\begin{equation}
%\frakN_{\rm ram}=\prod_{\substack{\frakp^e \parallel \frakN \\ \frakp \mid \frakD}} \frakp^e  \quad \text{and} \quad   \frakN_{\rm spl}=\prod_{\substack{\frakp^e \parallel \frakN \\ \frakp \nmid \frakD}} \frakp^e.
%\end{equation}
Then the even Clifford algebra $\calO \colonequals \Clf^0(\Lambda) \subseteq B$ of $\Lambda$ is a locally residually unramified quaternion $R$-order.  

Our main result, underlying \Cref{thm:mainthm} (proven as \Cref{thm:mainthmo}), is the following.  

\begin{thm} \label{thm:equiv}
For each sign vector $\varepsilon$, there is a Hecke-equivariant bijection 
\begin{equation}
S_k(\SO(\widehat{\Lambda}), \nu_\varepsilon) \xrightarrow{\sim} S_k(\widehat{\calO})^{\varepsilon}
\end{equation}
between the space of cuspidal orthogonal modular forms for $\Lambda$ with weight $k$ and character $\nu_\varepsilon$ and the space of quaternionic cusp forms on $\calO$ of weight $k$ with Atkin--Lehner eigenvalues $\varepsilon$.
\end{thm}

When $\calO$ is a suitable quaternion order, the Eichler--Shimizu--Jacquet--Langlands correspondence (recalled in \Cref{thm:ESJL}) then produces a final Hecke-equivariant bijection between the space of quaternionic forms and a space of Hilbert cusp forms.  (Restricting orders in this way does not lose any newforms, and it keeps the newform theory quite simple.)  
% (Our results could be extended to other orders, but the newform theory gets more complicated and we would not gain any newforms.)

\subsection*{Discussion}

The main result in \Cref{thm:equiv} could also be understood through analytic means via theta series or a Shimura correspondence for Hilbert modular forms.  This method of proof is quite challenging, requiring half-integral weight forms and twists to ensure nonvanishing.  By contrast, the proof using the even Clifford algebra is completely transparent---there is a natural bijection between the class set of $\Lambda$ and the type set of its even Clifford order $\calO$, and so the corresponding spaces of functions taking values in a representation are canonically isomorphic, equivariant with respect to the Hecke operators.  To highlight this simplicity, we present an overview of the proof in the case of $F=\Q$ in \ref{sec:classical}.  A key innovation in this paper is the use of the radical character allowing us to recover all forms, not just those invariant by the Atkin--Lehner operators.  

There are also other methods for computing with Hilbert modular forms, including quaternionic methods with either a definite or indefinite quaternion algebra: for an overview, see Demb\'el\'e--Voight \cite{DV}.  However, there are two key advantages of the approach using ternary orthogonal modular forms, both implying that we may work in smaller-dimensional spaces and therefore save substantially in the cost of linear algebra operations.
\begin{enumerate}
\item We work directly in the space with trivial central character, even when $F$ has nontrivial (narrow) class number.
\item We work directly in each Atkin--Lehner eigenspace.
\end{enumerate}
Our matrices are also as sparse as in the definite quaternion algebra case, but with a simpler reduction theory---working with ternary quadratic forms instead of quaternion ideals (whose reduced norm forms are quaternary quadratic forms).

The implementation in Magma \cite{Magma} of a very general version of the algorithm to compute orthogonal modular forms has been reported on by Assaf--Fretwell--Ingalls--Logan--Secord--Voight \cite{defortho}, based on an initial implementation described by Green\-berg--Voight \cite{gv}.  For this paper, we provide a complexity analysis and also implemented an optimized version in C++ restricted to $F=\Q$ and squarefree level $N$, available online \cite{code}.  As Birch notes, these algorithms perform very well in practice (see \cref{sec:examples})!

\subsection*{Organization}  

The paper is organized as follows.  In \cref{sec:classical}, to orient the reader we give a simplified overview and presentation of the method in the special case where $F=\Q$, recovering spaces of classical modular forms.  In \cref{sec:amfs}, we provide a brief setup for algebraic modular forms.  %  and prove a useful lemma which shows we can without loss of generality work with maximal level.  
We then specialize first to the case of definite quaternionic modular forms, relating them to Hilbert modular forms in \cref{sec:hmf}.  We then specialize to orthogonal modular forms and explain the functoriality of the even Clifford algebra relating ternary quadratic forms and quaternion orders in \cref{sec:ternarycliff}: this provides a natural, Hecke-equivariant relationship between the two spaces.  In \cref{sec:radchar}, we introduce the radical character.  We then discuss orthogonal newform theory in \cref{sec:newforms}, present the algorithms in \cref{sec:orthalg}, and finally report on a specialized implementation for $F=\Q$ in \cref{sec:examples}.  We conclude in \cref{AppendixA} with a categorical equivalence which gives an alternative and more general approach to the correspondence provided by the even Clifford functor.   

\subsection*{Acknowledgements}  

The authors would like to thank Eran Assaf, Asher Auel, Bryan Birch, Dan Fretwell, Benedict Gross, Adam Logan, Ariel Pacetti, Gustavo Rama, and Rainier Schulze-Pillot for their helpful comments over many years.  Voight was supported by grants from the Simons Foundation (550029 and SFI-MPS-Infrastructure-00008650).  

%%%%%%%%%%%%%%%%%%%%%%%%%%%%%%%%%%%%%%%%%%%%%%%%%%%%%%%%%%%%%%%%%%%%%%%%

\section{Overview in the classical case} \label{sec:classical}

In this section, we briefly describe our results over $F=\Q$, eventually restricting to the case of squarefree level $N$ and weight $k=2$, a case of special interest to Birch \cite{Birch}.  The reader can skip this section entirely or take it as an overview of the proof of our main result.  % We take advantage of the special setting and hopefully provide a good entry point into the general case treated in the main body.  

Let $Q \colon V \to \Q$ be a positive definite ternary (i.e., $\dim_\Q V=3$) quadratic space with associated bilinear form 
\begin{equation}
T(x,y) \colonequals Q(x+y)-Q(x)-Q(y) \quad \text{ for $x,y \in V$}. 
\end{equation}
Let $\Lambda \subset V$ be a lattice ($\Lambda \simeq \Z^3$ is the $\Z$-span of a $\Q$-basis for $V$) that is \defi{integral}, i.e., $Q(\Lambda) \subseteq \Z$.   Choosing a basis $\Lambda=\Z e_1+\Z e_2+\Z e_3 \simeq \Z^3$  gives a quadratic form 
\begin{equation} \label{eqn:Q}
Q_\Lambda(x e_1+ye_2+ze_3) = ax^2+by^2+cz^2+uyz+vxz+wxy \in \Z[x,y,z]  
\end{equation}
and vice versa.  Define the \defi{(half-)discriminant}
\begin{equation}
\begin{aligned}
N \colonequals \disc(\Lambda) =\disc(Q_\Lambda) & \colonequals \det(T(e_i,e_j))_{i,j}/2 = \frac{1}{2}\det\begin{pmatrix} 2a & w & v \\ w & 2b & u \\ v & u & 2c \end{pmatrix}  \\
&= 4abc+uvw-au^2-bv^2-cw^2  \in \Z_{>0} \end{aligned} 
\end{equation}

We define the orthogonal group
\begin{equation}
\OO(V) \colonequals \{g \in \GL(V) : Q(gx)=Q(x) \text{ for all $x \in V$}\}
\end{equation}
and define $\OO(\Lambda)$ similarly; we have $\#\OO(\Lambda)<\infty$.  
We say lattices $\Lambda,\Pi \subset V$ are \defi{isometric}, written $\Lambda \simeq \Pi$, if there exists $g \in \OO(V)$ such that $g\Lambda=\Pi$.  We make similar definitions over $\Q_p$.  
The \defi{genus} of $\Lambda$ is
\begin{equation} \Gen(\Lambda) \colonequals \{\Pi \subset V : \Lambda_p \simeq \Pi_p \text{ for all $p$}\}. 
\end{equation}
The \defi{class set} $\Cl(\Lambda)$ is the set of isometry classes in $\Gen(\Lambda)$.  We have $\#\Cl(\Lambda)<\infty$ by the geometry of numbers.

Kneser's theory of $p$-neighbors \cite{Kneser,Voight:Kneser} gives an effective method to compute the class set, and it also gives the Hecke action, as follows.
Let $p \nmid \disc(\Lambda)$ be prime.  We say that a lattice $\Pi \subset V$ is a \defi{$p$-neighbor} of $\Lambda$, and write $\Pi \sim_p \Lambda$,  if $\Pi$ is integral and
\begin{equation}
[\Lambda\colon \Lambda \cap \Pi]=[\Pi:\Lambda \cap \Pi]=p. 
\end{equation}
If $\Lambda \sim_p \Pi$, then $\Pi \in \Gen(\Lambda)$.  The set of $p$-neighbors is efficiently computable: $\Pi \sim_p \Lambda$ if and only if there exists $v \in \Lambda$, $v\not\in p\Lambda$ such that $Q(v) \equiv 0 \pmod{p^2}$ and
\[ \Pi = p^{-1}v + \{w \in \Lambda : T(v,w) \in p\Z\}.  \]
The line spanned by $v$ in $\Lambda/p\Lambda$ uniquely determines $\Pi$,  accordingly there are exactly $p+1$ neighbors $\Pi$
(see \cite[Theorem 3.5]{Tornaria-thesis}).

% Let $\Lambda=\Z^3=\Z e_1 + \Z e_2 + \Z e_3 \subset \Q^3$ have the quadratic form
% \[ Q_\Lambda(x,y,z) = x^2+y^2+3z^2+xz  \]
% and bilinear form given by
% \[ [T_\Lambda]=\begin{pmatrix} 2 & 0 & 1 \\ 0 & 2 & 0 \\ 1 & 0 & 6 \end{pmatrix}.  \]
% Thus
% \[ \disc(Q_\Lambda)=11.  \]
% We have $\#\Cl(\Lambda)=2$, with the nontrivial class represented by the $3$-neighbor
% \[ \Lambda' = \Z e_1 + 3\Z e_2 + \textstyle{\frac{1}{3}}\Z(e_1+2e_2+e_3)  \]
% with corresponding quadratic form
% \[ Q_{\Lambda'}(x,y,z) = x^2 + 9y^2 + z^2 + 4yz + xz. \]
% % > M*T*Transpose(M);
% % [ 2  0  1]
% % [ 0 18  4]
% % [ 1  4  2]
% \end{frame}

The space of \defi{orthogonal modular forms} for $\Lambda$ and trivial weight is
\begin{equation} 
M(\OO(\Lambda)) \colonequals \Map(\Cl(\Lambda),\C). 
\end{equation}
In the basis of characteristic functions for $\Lambda$, we have $M(\OO(\Lambda)) \simeq \C^{\#\!\Cl(\Lambda)}$.

For $p \nmid \disc(\Lambda)$, define the \defi{Hecke operator}
\begin{equation} \begin{aligned}
T_p \colon M(\OO(\Lambda)) &\to M(\OO(\Lambda)) \\
T_p(f)([\Lambda'])  & \colonequals \sum_{\Pi' \sim_p\, \Lambda'} f([\Pi']). 
\end{aligned} 
\end{equation}
The operators $T_p$ commute and are self-adjoint with respect to a natural inner product, so there is a basis of simultaneous eigenvectors, called \defi{eigenforms}.  The Hecke operators restrict to $S(\OO(\Lambda)) \subset M(\OO(\Lambda))$, the orthogonal complement of the constant functions.

To compute this action: there is an explicit reduction theory of integral ternary quadratic forms due to Eisenstein \cite{Eisenstein} and improved by Schiemann \cite[\S 2]{Schiemann}, generalizing Gauss reduction of integral quadratic forms, which allows us to uniquely identify (special) isometry classes.  

% In the running example with discriminant $11$,  we compute
% \[ [T_3]=\begin{pmatrix} 2 & 2 \\ 3 & 1 \end{pmatrix}, \quad
% [T_5]=\begin{pmatrix} 4 & 2 \\ 3 & 3 \end{pmatrix}, \quad \dots.  \]
% The constant function $e=\begin{pmatrix} 1 \\ 1 \end{pmatrix} \in M(\SO(\Lambda))$  is an \emph{Eisenstein series} with $T_p(e)=(p+1)e$.  

% Another eigenvector is $f=\begin{pmatrix} 2 \\ -3 \end{pmatrix}$ with $T_p(f)=a_p(f)$: 
% \[ a_3=-1, a_5=1, \dots, a_{11}=1.  \]
% We match it with the modular form
% \[ \sum_{n=1}^{\infty} a_n q^n = \prod_{n=1}^{\infty} (1-q^n)^2(1-q^{11n})^2=q - 2q^2 - q^3 + \ldots \in S_2(\widehat{\Gamma}_0(11)).  \]
% The \emph{Atkin--Lehner} involution $z \mapsto \displaystyle{\frac{-1}{11z}}$ acts on $f(z)\,\textup{d}z$ with eigenvalue $w_{11}=-a_{11}=-1$.  
Birch observed that there is an inclusion 
\begin{equation} 
S(\OO(\Lambda)) \hookrightarrow S_2(\Gamma_0(N)) 
\end{equation}
adding that it was ``provable'' \cite[p.~203]{Birch} and sketching an argument.  Indeed, this inclusion is explained by the even Clifford algebra.  
We define the \defi{even Clifford algebra} of $\Lambda$ by
\begin{equation} 
\calO \colonequals \Clf^0(\Lambda)  \colonequals \Z \oplus \Z i \oplus \Z j \oplus \Z k 
\end{equation}
with standard involution and multiplication laws
\begin{equation} 
\begin{aligned}
i^2 &= ui-bc \phantom{yourmom}   & jk &= a\overline{i} = a(u-i)  \\
j^2 &= vj-ac  & ki &= b\overline{j} = b(v-j) \\
k^2 &= wk-ab & ij &= c\overline{k} = c(w-k).
\end{aligned} 
\end{equation}
% so that e.g.\ $kj = \overline{\overline{j}\,\overline{k}}=-vw+ai+wj+vk$.  
Completing the square, we have 
\begin{equation}
\calO \subset \calO \otimes \Q \colonequals  B \simeq \quat{w^2-4ab,-aN}{\Q} \end{equation}
(with other similar symmetric expressions) so $\calO$ is an order in a definite quaternion algebra $B$.  Let $D \colonequals \disc(B)$ be the product of the primes $p$ that ramify in $B$.  

\begin{thm} \label{thm:QO}
The association $\Lambda \mapsto \calO=\Clf^0(\Lambda)$ is functorial and induces a (reduced) discrim\-inant-preserving bijection 
\begin{eqnarray} 
\left\{ \begin{minipage}{16ex} 
\begin{center}
Lattices $\Lambda \subset V$ \\
up to isometry
\end{center} 
\end{minipage}
\right\}  &\leftrightarrow &
\left\{ \begin{minipage}{26ex} 
\begin{center}
Quaternion orders $\calO \subset B$ \\
up to isomorphism
\end{center} 
\end{minipage}
\right\}. 
\end{eqnarray}
\end{thm}

\Cref{thm:QO} has a long history \cite[Remark 22.5.13]{Voight:quatbook}, with perhaps the earliest version going back to Hermite.  

Just as we defined the genus of $\Lambda$, we similarly define the genus of $\calO$ to be the set of orders in $B$ which are locally isomorphic to $\calO$; the set of (global) isomorphism classes in the genus is called the \defi{type set} $\Typ \calO$.  

\begin{cor} \label{cor:bij}
The even Clifford map induces a natural bijection
\begin{equation}  \label{eqn:ClTypO}
\Cl \Lambda \leftrightarrow \Typ \calO.
\end{equation}
\end{cor}

Of course from this corollary we obtain an isomorphism
\[ M(\OO(\Lambda)) = \Map(\Cl \Lambda,\C) \xrightarrow{\sim} \Map(\Typ(\calO),\C). \]

We now move from orders to ideals.  Define the \defi{(right) class set} of $\calO$ to be the set of locally principal fractional right $\calO$-ideals (those locally isomorphic to $\calO$) up to left multiplication by $B^\times$.  Then taking the left order gives a surjective map of (pointed) sets
\begin{equation} \label{eqn:ClsTyp}
\begin{aligned}
\Cls \calO &\to \Typ \calO \\
[I] &\mapsto [\calO_{\textup{\textsf{L}}}(I)] 
\end{aligned}
\end{equation}
which induces an injective $\C$-linear map
\begin{equation} 
\Map(\Typ(\calO),\C) \to M(\calO) \colonequals \Map(\Cls \calO,\C).
\end{equation}

Suppose now for simplicity that $N$ is \emph{squarefree}.  Then $\calO$ is an Eichler order (in fact, a hereditary order).  The fibers of \eqref{eqn:ClsTyp} are naturally identified with two-sided ideals, measured by the group $\AL(\calO) \simeq \prod_{p \mid N} C_2$ of Atkin--Lehner involutions.  By naturality and restricting to the cuspidal subspace, we have a Hecke-equivariant isomorphism
\begin{equation} 
S(\OO(\Lambda)) \xrightarrow{\sim} S(\calO)^{+}
\end{equation}
where $S(\calO)^+$ are those forms invariant under $\AL(\calO)$.  Finally, the Jacquet--Langlands correspondence gives
\begin{equation}
S(\calO)^+ \hookrightarrow S_2^{\textup{$D$-new}}(\Gamma_0(N))^{\delta}
\end{equation}
where for $M \mid N$ we write $\varepsilon_M$ for the sign vector which is $-1$ exactly at the primes $p \mid M$ and then $\delta \colonequals \varepsilon_D$ for the Hilbert symbol for $B$.

To get all forms, we add a representation.  We equip $V$ with an orientation (e.g., choose an ordered $\Q$-basis $V \simeq \Q^3$) and define $\SO(V) \colonequals \OO(V) \cap \SL(V)$.  We note that $\Orth(V) \simeq \{\pm 1\} \times \SO(V)$.  Therefore, we have a natural homomorphism obtained from the composition 
\begin{equation} \nu\colon \Orth(V)  \to \Orth(V)/\{\pm 1\}  \simeq \SO(V)  \simeq B^\times/\Q^\times  \xrightarrow{\nrd} \Q^\times/\Q^{\times 2} 
\end{equation}
called the \defi{spinor norm}.  A delightful calculation (see \Cref{lem:spinornorm}) shows that if $\gamma \in \SO(V)$ has $\tr(\gamma)\neq -1$, then $\gamma(\nu)=\tr(\gamma)+1$. 

For $r \mid N$, we define the \defi{spinor character} for $r$ by
\begin{equation} 
\begin{aligned}
\nu_r: \Orth(V) \xrightarrow{\nu} \Q^\times/\Q^{\times 2}  &\to \{\pm 1\}  \\
a &\mapsto \prod_{p \mid r} (-1)^{\ord_p(a)} 
\end{aligned} 
\end{equation}
The spinor characters ``find'' the remaining forms with all Atkin--Lehner signs.  We let $M(\SO(\Lambda),\nu_r)$ be the space of functions from $\Gen \Lambda$ to $\C$ that transform by $\nu_r$: that is, $f(\gamma \Lambda')=\nu_r(\gamma)f(\Lambda')$ for all $\gamma \in \SO(V)$.  We take again $S(\SO(\Lambda),\nu_r)$ the orthogonal complement of the constant functions: when $r \neq 1$, we have $S(\SO(\Lambda),\nu_r)=M(\SO(\Lambda),\nu_r)$.  

Putting all of this together, we obtain \Cref{thm:equiv} in this special case: we have a Hecke-equivariant isomorphism
\begin{equation} 
S(\SO(\Lambda),\nu_r) \simeq S_2^{\text{$D$-new}}(\Gamma_0(N))^{\delta\varepsilon_r}.
\end{equation}

To prove \Cref{thm:mainthm}, we compute the Hecke action on $S(\SO(\Lambda),\nu_r)$ using ternary forms and $p$-neighbors, giving the Hecke action on the space of classical modular forms.  To compute $T_p$ on $S(\SO(\Lambda),\nu_r)$ for $p \nmid N$, starting with $\Lambda$ we iteratively compute $p$-neighbors and reduce to identify isomorphism classes: we repeat this $O(\#\Cl \Lambda)=O(d)$ times with $p+1$ neighbors, so the running time is $\widetilde{O}(pd)$ where 
\begin{equation} 
d \colonequals \dim S(\SO(\Lambda),\nu_r)
\end{equation}
is the dimension of the output space.

\begin{rmk}
Birch says that his original motivation was to generalize the \emph{method of graphs} due to Mestre--Oesterl\'e.  This method was recently studied by Cowan \cite{Cowan}, so we briefly explain the connection here as this method provides an extension to nonsquare level with many parallel aspects.  Let $\calO$ be a maximal order in a quaternion algebra $B$ with discriminant $\disc(B)=p$.  Then there are at most two supersingular curves $E$ up to isomorphism over $\F_p^{\textup{al}}$ such that $\End(E) \simeq \calO$; there are two such if and only if $j(E) \in \F_{p^2} \smallsetminus \F_p$ if and only if the unique two-sided ideal of $\calO$ of reduced norm $p$ is not principal.  So if we identify $E$ with its image $E^{(p)}$ under the Frobenius map when $j(E) \not\in \F_p$, then we get a bijection of such classes with the type set $\Typ \calO$, which is in bijection with ternary quadratic forms of discriminant $p$ (also known as the Deuring correspondence).  
\end{rmk}

We conclude with a simple example.  

\begin{exm}
Let $\Lambda=\Z^3=\Z e_1 + \Z e_2 + \Z e_3 \subset \Q^3$ have the quadratic form
\[ Q_\Lambda(x,y,z) = x^2+y^2+3z^2+xz  \]
and bilinear form given by
\[ [T_\Lambda]=\begin{pmatrix} 2 & 0 & 1 \\ 0 & 2 & 0 \\ 1 & 0 & 6 \end{pmatrix} \]
so $\disc(Q_\Lambda)=11$.  We compute that among the $2$-neighbors of $\Lambda$, one is isometric to $\Lambda$ and the other two reduce to a new lattice $\Lambda'$ with Gram matrix
\[ \begin{pmatrix} 2 & -1 & 1 \\ -1 & 2 & -1 \\ 1 & -1 & 8 \end{pmatrix}. \]
Among the $2$-neighbors of $\Lambda'$, all $3$ are isometric to $\Lambda$.  This yields 
\[ [T_2] = \begin{pmatrix} 1 & 2 \\ 3 & 0 \end{pmatrix}. \]
Similarly, we compute
\[ [T_3]=\begin{pmatrix} 2 & 2 \\ 3 & 1 \end{pmatrix}, \quad
[T_5]=\begin{pmatrix} 4 & 2 \\ 3 & 3 \end{pmatrix}, \quad \dots.  \]
The constant function $e=\begin{pmatrix} 1 \\ 1 \end{pmatrix} \in M(\SO(\Lambda))$  is an Eisenstein series with $T_p(e)=(p+1)e$.  
Another eigenvector is $f=\begin{pmatrix} 2 \\ -3 \end{pmatrix}$ with $T_p(f)=a_p(f)$: 
\[ a_3=-1, a_5=1, \dots  \]
We match it with the familiar modular form \href{https://www.lmfdb.org/ModularForm/GL2/Q/holomorphic/11/2/a/a/}{\textsf{11.2.a.a}}:
\[ \sum_{n=1}^{\infty} a_n q^n = \prod_{n=1}^{\infty} (1-q^n)^2(1-q^{11n})^2=q - 2q^2 - q^3 + \ldots \in S_2(\Gamma_0(11)).  \]
And indeed the Atkin--Lehner involution $z \mapsto \displaystyle{\frac{-1}{11z}}$ acts on $f(z)\,\textup{d}z$ with eigenvalue $w_{11}=-a_{11}=-1$, as expected.
\end{exm}

%%%%%%%%%%%%%%%%%%%%%%%%%%%%%%%%%%%%%%%%%%%%%%%%%%%%%%%%%%%%%%%%%%%%%%%%

\section{Algebraic modular forms} \label{sec:amfs}

We now begin the general case.  In this section, we briefly review the theory of algebraic modular forms, as our quaternionic modular forms and orthogonal modular forms will be special cases of this theory.  For further reference, see Gross \cite{Gross1}.  Concrete versions of the general theory in this section will be provided in the two sections that follow.

Let $F$ be a number field with ring of integers $R \colonequals \Z_F$.  Let $\bdG$ be a reductive algebraic group defined over $F$.  % Then $\bdG$ is a closed algebraic subgroup of the algebraic group $\GL_{n,F}$ for some $n \geq 1$, defined by polynomials in the entries and the inverse of the determinant of a matrix of indeterminates; the \emph{reductive} condition (the maximal connected unipotent normal subgroup is trivial) is technical and its role will remain invisible below.  % The most basic examples of reductive algebraic groups are the groups $\GL_n$ themselves, but of special interest for us will be the orthogonal $\OO_n$ and special orthogonal $\SO_n$ groups, and inner forms of these groups.  
Let $G_\infty \colonequals \prod_{v\mid \infty} \bdG(F_v)$.  Then $G_\infty$ is a real Lie group with finitely many connected components.  We consider a very special case: suppose that $G_\infty$ is compact modulo its center.  This setup was investigated by Gross \cite{Gross1}, who showed that modular forms and automorphic representations can be fruitfully studied in the absence of analytic hypotheses; in particular, we are in an ideal setting to develop algorithms.

Let $\Fhat \colonequals \tprodprime{\frakp} F_\frakp$ be the finite adeles of $F$.  Let $\Khat$ be a compact open subgroup of $\Ghat \colonequals \bdG(\Fhat)$, let $G \colonequals\bdG(F)$, and consider the double coset space
\begin{equation} 
Y \colonequals G \backslash \Ghat/\Khat. 
\end{equation}
The set $Y$ is finite (being both discrete and compact), and so we can write a finite decomposition
\begin{equation} \label{eqn:decomp}
\Ghat=\bigsqcup_i^{<\infty} G \betahat_i \Khat. 
\end{equation}
Each point $G \betahat_i \Khat\in Y$ has a (finite) stabilizer group
\begin{equation} 
\Gamma_i \colonequals \betahat_i \Khat \betahat_i^{-1} \cap G
\end{equation}
whose conjugacy class in $G$ is independent of the choice of representative $\betahat_i$.

Let $W$ be a finite-dimensional representation of $G$ over $\C$ with a left action $x \mapsto g\cdot x$ for $g \in G$ and $x \in W$.  The space of \defi{modular forms on $\bdG$ of level $\Khat$ and weight $W$} is
\begin{equation} 
M_W( \Khat) \colonequals \{ f\colon \Ghat \to W : f(g\alphahat \uhat) = g \cdot f(\alphahat) \text{ for all
$g \in G$, $\uhat \in \Khat$}\}
\end{equation}
A form $f \in M_W(\Khat)$ is locally constant on $\Ghat$ since it is invariant under the open compact $\Khat$.  % (We could equally well work with this kernel as the level structure, but we find this in better analogy with how classical modular forms are treated.)

With respect to a decomposition \eqref{eqn:decomp}, we see that as $\C$-vector spaces, we have an isomorphism
\begin{equation} \label{eqn:decompbetahati}
\begin{aligned} 
M_W(\Khat) &\xrightarrow{\sim} \bigoplus_i H^0(\Gamma_i, W) \\
f &\mapsto (f(\betahat_i))_i
\end{aligned}
\end{equation}
where $H^0(\Gamma_i,W)$ denotes $\Gamma_i$-invariants, since $g\betahat_i \Khat = \betahat_i \Khat$ if and only if $g \in \Gamma_i$.  
% \jv{We could write $W^{\Gamma_i}$ instead, but I'm just being careful right now; and for the application of Shapiro's lemma, it is natural to write it this way as it is then seen as a special case $i=0$.}

We will also make use of two slight generalizations and modifications.  First, we allow character.  Suppose that $\Khat \trianglelefteq \Khat_0$ is a normal inclusion such that $C \colonequals \Khat_0/\Khat$ is \emph{abelian}.  Since $\Khat_0$ normalizes $\Khat$, the group $C$ acts by right translation on $\Ghat/\Khat$, inducing operators on $M_W(\Khat)$ which we call \defi{diamond operators} by analogy with the classical case.  Then we have an isotypic decomposition
\begin{equation} \label{eqn:decomposechar}
M_W(\Khat) \xrightarrow{\sim} \bigoplus_{\clap{$\scriptstyle
  \chi \in C\spcheck\!\!\!\!$}} M_W(\Khat_0,\chi) 
\end{equation}
according to characters $\chi \colon C \to \C^\times$, where we define
\begin{equation} 
M_W(\Khat_0,\chi) \colonequals \{ f\colon \Ghat \to W : f(g\alphahat \uhat) = \chi(u)(g \cdot f(\alphahat)) \text{ for all
$g \in G$, $\uhat \in \Khat_0$}\}.
\end{equation}

Second, similarly the normalizer mod centralizer $A \colonequals N_{\Ghat}(\Khat)/(\Khat C_{\Ghat}(\Khat))$ acts by conjugation on $\Khat$ and hence on $M_W(\Khat)$. When $A$ is abelian we have an isotypic decomposition
\begin{equation} \label{eqn:decomposecharbig}
M_W(\Khat) \xrightarrow{\sim} \bigoplus_{\clap{$\scriptstyle
  \chi \in A\spcheck\!\!\!\!$}} M_W(\Khat)^{\chi} 
\end{equation}
where the subspace is defined by the condition that $A$ acts via $\chi$.

\begin{rmk} \label{rmk:reps}
As further generalizations, we could equally well work with $W$ a finite-dimensional representation over any field, and weaken the requirement that the groups $C=\Khat_0/\Khat$ and $A=N_{\Ghat}(\Khat)/(\Khat Z_{\Ghat}(\Khat))$ are abelian by instead taking the isotopic decomposition indexed instead by irreducible representations.
\end{rmk}

The space $M_W(\Khat)$ is equipped with the action of Hecke operators $T_{\pihat}$ for $\pihat \in \Ghat$, depending on the class $\Khat \pihat \Khat$, as follows: write
\begin{equation}
\Khat \pihat \Khat = \bigsqcup_j \, \pihat_j \Khat 
\end{equation}
as a finite disjoint union, and then for $f \in M_W(\Khat)$ define
\begin{equation} 
(T_{\pihat} f)(\alphahat \Khat) \colonequals \sum_j f(\alphahat \pi_j \Khat). 
\end{equation}
This is well-defined (independent of the choice of representative $\pihat$ and representatives $\pihat_j$) by the right $\Khat$-invariance of $f$.

We call the subspace of $M_W(\Khat)$ of functions that factor through $\Ghat/[\Ghat,\Ghat]$ the space of \defi{Eisenstein series}.  There is a natural inner product on $M_W(\Khat)$---the dot product weighted by $1/\#\Gamma_i$---for which the Hecke operators are self-adjoint, and we let the \defi{cuspidal subspace} be the orthogonal complement of the Eisenstein series under the inner product.

\section{Hilbert modular forms as quaternionic modular forms}
\label{sec:hmf}
%\gt{Definition, notation, references, explain quaternionic modular
%forms with arbitrary weight representation. [+coinduced?]}

In this section, we specialize the previous section to the case of groups from definite quaternion algebras.  

Let $B$ be a definite quaternion algebra over a totally real field $F$.  We take $\bdG$ to be the reductive group associated to $G=B^\times$, so that $G_\infty=B_\infty^\times \simeq \prod_{v \mid \infty} \HH^\times$  where $\HH$ is the division ring of real Hamiltonians, and $\widehat{G} = \widehat{B}^\times$.  

For level structure, we let $\calO \subseteq B$ be an $R$-order and take $\Khat=\calOhat^\times$; let $\frakN \colonequals \discrd(\calO)$.  Then the double coset space
\begin{equation} 
Y = G\backslash \Ghat / \Khat = B^\times \backslash \Bhat^\times\!
/\calOhat^\times \leftrightarrow \Cls \calO 
\end{equation}
is naturally identified with the \defi{(right) class set} of $\calO$ \cite[Lemma 27.6.8]{Voight:quatbook}.  
% As a second level structure, we take $\Khat=N_{\Bhat^\times}(\calOhat)^\times/\Fhat^\times$ the adelic normalizer group.  Now the double coset space is naturally identified with the \defi{type set} $\Typ \calO$ \cite[27.6.23]{Voight:quatbook} and moreover the natural surjective map $\Cls \calO \to \Typ \calO$ is the map obtained by $[I] \mapsto [\calO_{\textsf{L}}(I)]$ taking the left order.  

We restrict to weights $k=(k_v)_v \in \prod_{v \mid \infty} 2\Z_{>0}$, corresponding to the representation $W_k \colonequals \bigotimes_{v \mid \infty} \Sym^{k_v-2}(\C)(-k_v/2-1)$ with respect to a splitting $B^\times \hookrightarrow \GL_2(\C)$, but we also consider twists of these weights by characters.  On the associated space of modular forms, there is an action by diamond operators corresponding to the class group \cite[41.3.4]{Voight:quatbook} which picks out the central character; we let $M_k(\widehat{\calO}^\times)$ be the fixed space.  The Eisenstein subspace consists of those functions that factor through the norm map $\nrd\colon \hatx B \to \hatx F$; this space is trivial unless $k_v=2$ for all $v \mid \infty$.  % , in which case it consists of maps $\Cl_{G(\calO)} R \colonequals F_{>0}^\times \backslash \Fhat^\times / \nrd(\calOhat)^\times \to \C$, which are maps on the class group $\Cl_{G(\calO)} R$ associated to $\calO$ \cite[27.7.4]{Voight:quatbook}.  ; 
We let $S_k(\widehat{\calO}^\times) \subseteq M_k(\widehat{\calO}^\times)$ be the orthogonal complement of the Eisenstein subspace.  

% For each $v\mid\infty$, since $B_v \simeq \mathbb{H}$ we have $B_v^\times/F_v^\times\simeq \mathop{\rm SO}(3)$, hence it has a unique irreducible representation $W_v$ of dimension $k_v-1$ with $k_v$ even. For instance, one can take $W_v$ as the space of harmonic homogenous polynomials in $B_v/F_v$ of degree $k_v/2-1$.
% Combining these, let $W_k \colonequals \bigotimes_{v\mid\infty} W_v$ as a representation of $B^\times/F^\times \hookrightarrow \otimes_{v\mid\infty} B_v^\times/F_v^\times$.

% The space of \defi{quaternionic modular forms} of weight $k$ and level $\frakN$ is defined as
% $M_k(\calO) = M(\hatx\calO,W_k)$.
% The \defi{cuspidal subspace} $S_k(\calO)$ is the orthogonal complement of the subspace of functions that factor through the norm map $\nrd:\hatx B\to\hatx F$. Note $S_k(\calO)=M_k(\calO)$ unless $k_v=2$ for all $v\mid\infty$.  The codimension of $\calS_2(\calO)$ in $M_2(\calO)$ is the narrow class number of $F$.  \jv{Give reference.}
% \gt{not true for trivial central character, fix or remove}

We further define the group 
\begin{equation} \label{eqn:ALOHat}
\AL(\calOhat) \colonequals N_{\Bhat^\times}(\calOhat)/(\Fhat^\times \calOhat^\times)  \simeq \prod_{\frakp \mid \frakN} \AL(\calO_\frakp) 
\end{equation}
which we call the adelic \defi{Atkin--Lehner group}.  (Although there are complications in the global normalizer group \cite[\S 28.9]{Voight:quatbook}, they do not arise adelically.)  Then as in \eqref{eqn:decomposecharbig} when $\AL(\calOhat)$ is abelian---and more generally as in \Cref{rmk:reps}---we have
\begin{equation}
S_k(\calOhat^\times) \simeq \bigoplus_{\chi \in \AL(\calOhat)\spcheck} S_k(\calOhat^\times)^\chi
\end{equation}
checking that the Eisenstein subspace is preserved under the normalizer.

We now focus on a very general class of orders for which we have good control of the normalizers.  See Voight \cite[\S 24.3]{Voight:quatbook} for background reading.  Let $\frakN\subseteq R$ be a nonzero ideal.

\begin{defn}
A quaternion algebra $B$ over $F$ is \defi{suitable} for $\frakN$ provided that
\begin{enumroman}
    \item $B$ is ramified at all infinite places; and
    \item if $B$ is ramified at $\frakp$, then $v_\frakp(\frakN)$ is odd.
\end{enumroman}
\end{defn}

\begin{proposition}
There are quaternion algebras suitable for $\frakN$ \emph{except} when $[F:\Q]$ is odd and $\frakN$ is a square.
\end{proposition}
\begin{proof}
When $[F:\Q]$ is even, we can take $B$ ramified at all the infinite places and split at all the finite places.
When $[F:\Q]$ is odd and $\frakN$ is not a square, we can take $B$ ramified at all infinite places and at one $\frakp$, where $v_\frakp(\frakN)$ is odd, and split at all other finite places.
\end{proof}

\begin{remark}
We will not assume that the quaternion algebra is as given by the proof above.  For reasons that will be clear later, it will be more convenient to have a quaternion algebra ramified at more primes.
\end{remark}

An order $\calO$ is \defi{locally residually unramified} \cite[24.3]{Voight:quatbook} if for all primes $\frakp \mid \discrd(\calO)$, the quotient $\calO_\frakp/\rad(\calO_\frakp)$ of the quotient of the completed order by the Jacobson radical is either $\F_\frakp \times \F_\frakp$ (\defi{residually split} or \defi{Eichler} at $\frakp$) or $\F_{\frakp^2}$ (\defi{residually inert} at $\frakp$), corresponding to the Eichler symbol values $1,-1$.  

We will need to recall some facts in the residually inert case, studied by Pizer \cite[\S2]{Pizer} (see also \cite[24.3.7]{Voight:quatbook}).  Let $\calO_\frakp$ be residually inert.  Lifting the extension $\calO_\frakp/\rad \calO_\frakp \simeq \F_{\frakp^2}$ we have a quadratic unramified extension $S_\frakp \supseteq R_\frakp$, the valuation ring in the quadratic unramified extension $K_\frakp \supseteq F_\frakp$.  Diagonalizing, we find $\calO_\frakp \simeq \quat{S_\frakp,b}{R_\frakp}$ for some $b \in R \smallsetminus \{0\}$.  Since $\Nm_{S_\frakp|R_\frakp}(S_\frakp^\times)=R_\frakp^\times$, we may suppose that $b=\pi^e$ for some $e \geq 0$ with $\pi$ a uniformizer for $R_\frakp$, and then $\discrd \calO_\frakp = \frakp^e$.  We have $e$ even when $B_\frakp \simeq \M_2(F_\frakp)$ and $e$ odd when $B_\frakp$ a division algebra \cite[24.4.9]{Voight:quatbook}.  Then \cite[Proof of Proposition 24.4.7]{Voight:quatbook} the order $\calO_\frakp \otimes S_\frakp \hookrightarrow \M_2(S_\frakp)$ is residually split, and the regular representation identifies
\begin{equation} \label{eqn:Opembed}
\calO_\frakp \hookrightarrow \left\{\begin{pmatrix} u & v \\ b\overline{v} & \overline{u} \end{pmatrix} : u,v \in S_\frakp\right\} \subseteq \M_2(S_\frakp)
\end{equation}
with $j \mapsto \begin{pmatrix} 0 & 1 \\ b & 0 \end{pmatrix}$.  

\begin{proposition} \label{prop:unique}
If $B$ is a suitable quaternion algebra for $\frakN$,
then there exists a residually unramified order $\calO \subseteq \frakN$ with $\discrd(\calO)=\frakN$.  
Moreover, the genera of such orders are determined by the choice of residually inert or split for $\frakp^e \parallel \frakN$ with $e$ even.  
\end{proposition}

We could make a unique choice in \Cref{prop:unique} by taking $\calO_\frakp$ residually split whenever $B_\frakp$ is split.  

\begin{proof}
As we have seen, all residually unramified orders in a local division quaternion algebra have odd exponent in their reduced discriminant, which is precisely guaranteed by the algebra being suitable; there is no condition at primes $\frakp$ split in $B$.  
% The local conjugacy class of a residually unramified order is determined by the Eichler invariants, and existence is subject to $\prod_{p\mid\frakN} e(\calO_p)^{v_p(N)} = (-1)^{[F:\Q]}$.  % We will not consider this generalization for now.  % One reason to avoid this for now is that we lose \Cref{prop:unique}.  
See also Gross \cite[Proposition 3.4]{Gross}.  
\end{proof}

\begin{remark}
Gross \cite{Gross} and Martin \cite{Martin} also consider residually ramified orders with bounded level, so our methods could be extended to that case too (e.g., the Lipschitz order in the Hurwitz order).  
\end{remark}

We are now ready to state the correspondence between Hilbert modular forms and quaternionic modular forms.  Let $B$ be a quaternion algebra suitable for $\frakN$, let $\frakD=\disc B$, and let $\calO$ be a residually unramified order with $\discrd(\calO)=\frakN$.  Write $\frakN = \frakN_{-} \frakN_{+}$ where
$\frakN_{-}$ is divisible by exactly those primes where $\calO$ is residually inert (i.e., Eichler invariant $-1$) and $\frakN_{+}$ by primes where $\calO$ is residually split (Eichler invariant $+1$).  Whenever $\frakp \mid \frakD$, we must have $\frakp \mid \frakN_{-}$.  
% If $\calO$ is suitable, then $v_\frakp(\frakN)$ is odd for all $\frakp\mid\frakN_{\rm ram}$.  

For Hilbert modular forms, we take adelic level structure $\widehat{\Gamma}_0(\frakN) \leq \GL_2(\widehat{R})$ writing the associated space of cuspidal Hilbert modular forms (with trivial central character) as $S_k(\widehat{\Gamma}_0(\frakN))$.  For $\frakM\mid\frakN$, let $S_k^{\frakM\text{-}\rm{new}}(\widehat{\Gamma}_0(\frakN)) \subseteq S_k(\widehat{\Gamma}_0(\frakN))$ be the subspace consisting of forms which are $\frakp$-new for all $\frakp\mid\frakM$ with respect to the natural degeneracy operators.  

\begin{theorem}[Eichler--Shimizu--Jacquet--Langlands correspondence] \label{thm:ESJL}
Let $\calO$ be a suitable quaternion order of level $\frakN$.
Then there is an isomorphism of Hecke-modules
\begin{equation} \label{eqn:JL}
    S_k(\widehat{\calO}) \simeq
    \bigoplus_{\frakM}
    S_k^{\textup{$\frakM$-new}}(\widehat{\Gamma}_0(\frakM\frakN_{+}))
\end{equation}
where the sum runs over all divisors $\frakM \mid \frakN_{-}$ such that $v_\frakp(\frakM)$ and $v_\frakp(\frakN)$ have the same parity for all $\frakp\mid\frakN_{-}$.
\end{theorem}

\begin{proof}
See Jacquet--Langlands~\cite[Chap. XVI]{jacqlang}, Gelbart--Jacquet \cite[\S 8]{geljac}, and Hida \cite[Proposition 2.12]{HidaCM}, and more recently Martin \cite[Theorem 1.1(ii)]{Martin}.
\end{proof}

% \gt{Talk about the normalizer, trivial central character, atkin lehner.
% \\Here we will define $\mu$, talk about $\mu/\hatx F = \prod_{\frakp\mid\frakN} \mu_\frakp$, state that for our orders $\mu_p$ is a group of order 2 generated by $\omega_\frakp$ with $\nrd(\omega_\frakp)=\frakp$}

% The normalizer group 
% $\mu = N_{\widehat{B}^\times}(\widehat{\calO})/\widehat{F}^\times \widehat{\calO}^\times$
% acts naturally on $S_k(\calO)$.
% Note that $\mu = \prod_\frakp \mu_\frakp$
% where
% $\mu_\frakp = N_{B_\frakp^\times}(\calO_\frakp)/F_\frakp^\times \calO_\frakp^\times$.
For these nice lattices and orders, we have a simple description of the normalizer.

\begin{prop} \label{prop:alltheAL}
Suppose that $\calO$ is residually unramified at $\frakp$.  Then the following statements hold.
\begin{enumalph}
\item
The Atkin--Lehner group $\AL(\calO_\frakp)$ has order $2$, generated by $w_\frakq$ where $\frakq=\frakp^e \parallel \frakN$.
% where $w_\frakp\in N_{B_\frakp^\times}$ with $\nrd w_\frakp=\frakp$.
%\[
%N_{\widehat{B}^\times}(\widehat{\calO})/\widehat{F}^\times \widehat{\calO}^\times 
%\prod_{\frakp\mid\frakN}
%N_{B_\frakp^\times}(\calO_\frakp)/F_\frakp^\times \calO_\frakp^\times 
%= \prod_{\frakp \mid \frakN} \langle w_\frakp \rangle
%\simeq \prod_{\frakp \mid \frakN} C_2
%\]
\item Under the correspondence with Hilbert modular forms \eqref{eqn:JL}, the involution $w_\frakq$ on $S_k(\calOhat^\times)$ corresponds to either the Atkin--Lehner involution on $S_k(\widehat{\Gamma}_0(\frakN))$ or its negative, according as $\frakp \mid \frakD$ or not.  
\end{enumalph}
\end{prop}

\begin{proof}
We begin with (a).  When $\calO_\frakp$ is residually split (Eichler) the result is well-known \cite[Proposition 23.4.14]{Voight:quatbook}.  The residually inert case can be derived from this, as follows.  The inclusion \eqref{eqn:Opembed} provides an inclusion $N_{B_\frakp^\times}(\calO_\frakp) \hookrightarrow N_{B_{K_\frakp}^\times}(\calO_{S_\frakp})$.  We check that indeed $j \in N_{B_\frakp^\times}(\calO_\frakp)$, and the rest is $\calO_{S_\frakp}^\times \cap B_\frakp^\times = \calO^\times$.  

Part (b) follows from work of Jacquet--Langlands \cite[Proposition 15.5]{jacqlang}, see also \cite[Theorem (8.1)]{geljac} Gelbart--Jacquet).  The result may also be derived via theta series using work of B\"ocherer--Schulze-Pillot \cite[Lemma 8.2]{BSP}.
\end{proof}

As in the introduction, a \defi{sign vector} (for $\frakN$) is 
\begin{equation}
\varepsilon = (\varepsilon_\frakp)_\frakp \in \prod_{\frakp^e \parallel \frakN} \{\pm 1\}.
\end{equation}
For a sign vector, we define
\begin{equation}
S_k(\widehat{\calO})^\varepsilon =
\{ f\in S_k(\widehat{\calO}) : w_\frakq f = \varepsilon_\frakp f \}
\end{equation}

We conclude this section with newform theory, with attention to Atkin--Lehner eigenvalues.  The algorithmic application will be to knowing precisely when we have carved away the old subspace.
%\begin{corollary}
%\end{corollary}

%\gt{See Roberts--Schmidt book in p.4 ``Theorem ($\GL(2)$ Hecke Eigenvalues and L-functions)'' and ``Theorem ($\GL(2)$ Oldforms Theorem)''.  To get Atkin--Lehner invariant level-raising operators, we need to take sums and differences.  I'm not sure we can use this as a reference, since no proof (or reference) is given.  But anyway, I think we know what we need to write here!}

\begin{defn}
Let $\alpha$ be the multiplicative function on divisors of $\frakN$ defined by
\[
\alpha(\frakp^e) =
\begin{cases}
1 + e, & \text{if $\frakp\mid\frakN_{+}$;} \\
1, & \text{if $\frakp\mid\frakN_{-}$ and $e$ is even; and} \\
0, & \text{if $\frakp\mid\frakN_{-}$ and $e$ is odd.} \\
\end{cases}
\]
\end{defn}

The function $\alpha$ records multiplicities in the newform decomposition as follows.  
 
\begin{theorem} \label{thm:skOtimes}
We have
\[
    S_k(\widehat{\calO}^\times) \simeq
    \bigoplus_{\frakM\mid\frakN}
    S_k^{\rm new}(\widehat{\Gamma}_0(\frakM))^{\oplus \alpha(\frakN/\frakM)}
\]
\end{theorem}

\begin{proof}
First, we rewrite \Cref{thm:ESJL} as 
\begin{equation} \label{eqn:skmnew0}
    S_k(\widehat{\calO}^\times) \simeq
    \bigoplus_{\frakM_{-}\mid\frakN_{-}}
    S_k^{\textup{$\frakM_{-}$-new}}(\widehat{\Gamma}_0(\frakM_{-}\frakN_{+}))^{\oplus \alpha(\frakN_{-}/\frakM_{-})}.  
\end{equation}
Second, we apply the usual Atkin--Lehner theory for Hilbert modular forms, see Roberts--Schmidt \cite[p.~4]{RobertsSchmidt} for a succinct summary.  We obtain
\begin{equation} \label{eqn:skmnew}
    S_k^{\text{$\frakM_{-}$-new}}(\widehat{\Gamma}_0(\frakM_{-}\frakN_{+}))
    \simeq
    \bigoplus_{\frakM_{+}\mid\frakN_{+}}
    S_k^{\rm new}(\widehat{\Gamma}_0(\frakM_{-}\frakM_{+}))^{\oplus \alpha(\frakN_{+}/\frakM_{+})}.
\end{equation}
% which is classical Atkin-Lehner theory (maybe cite [RS] book intro, note this second lemma is just a statement about Hilbert MF, no quaternions here; see Roberts--Schmidt book in p.4 ``Theorem ($\GL(2)$ Hecke Eigenvalues and L-functions)'' and ``Theorem ($\GL(2)$ Oldforms Theorem)''.)
We conclude by substituting \eqref{eqn:skmnew} into \eqref{eqn:skmnew0}, using multiplicativity of $\alpha$, and rewriting divisors as $\frakM = \frakM_{-}\frakM_{+} \mid \frakN_{-}\frakN_{+}=\frakN$.
\end{proof}

Now we refine this theorem to take Atkin--Lehner signs into account.  For $\frakp \mid \frakN_{-}$, there is no change in the Atkin--Lehner signs under degeneracy operators; for $\frakp \mid \frakN_{+}$, the possible Atkin--Lehner signs are as balanced as possible.  

\begin{defn}
For a sign vector $\varepsilon$, let $\alpha^\varepsilon$ be the multiplicative function on divisors of $\frakN$ such that
\[
\alpha^\varepsilon(\frakp^e) =
\begin{cases}
\left\lceil\frac{1 + e}2\right\rceil, & \text{if $\frakp \mid \frakN_+$ and $\varepsilon_\frakp = +1$;} \\
\left\lfloor\frac{1 + e}2\right\rfloor, & \text{if $\frakp \mid \frakN_+$ and $\varepsilon_\frakp = -1$;} \\
1, & \text{if $\frakp\mid\frakN_{-}$, $e$ is even, and $\varepsilon_\frakp=+1$; and} \\
0, & \text{if $\frakp\mid\frakN_{-}$, otherwise} \\
\end{cases}
\]
\end{defn}

Considering cases, it is straightforward to show that $\sum_\varepsilon \alpha^\varepsilon = \alpha$.
%\begin{lemma}
%We have $\sum_\varepsilon \alpha^\varepsilon = \alpha$.
%\end{lemma}
%\begin{proof}
%Immediate.
%% Clear (I think). Note in any case
%% $\alpha^\varepsilon(\frakp^e) = \lceil\alpha(\frakp^e)/2\rceil$ if $\varepsilon_\frakp=+1$ and 
%% $\alpha^\varepsilon(\frakp^e) = \lfloor\alpha(\frakp^e)/2\rfloor$ if $\varepsilon_\frakp=-1$. 
%\end{proof}

\begin{prop} \label{prop:ALSK}
Let $\varepsilon$ be a sign vector.  Then
\begin{equation} \label{eqn:Skimis}
    S_k^{\frakM_{-}\textup{-}\rm new}(\widehat{\Gamma}_0(\frakM_{-}\frakN_{+}))^\varepsilon
    \simeq
    \bigoplus_{\frakM_{+}\mid\frakN_{+}}
    \bigoplus_{\varepsilon'}
    \bigl(S_k^{\rm new}(\widehat{\Gamma}_0(\frakM_{-}\frakM_+))^{\varepsilon'}\bigr)^{\oplus \alpha^{\varepsilon\varepsilon'}(\frakN_{+}/\frakM_{+})}
\end{equation}
the sum over all sign vectors $\varepsilon'$.
\end{prop}

\begin{proof}
This is standard newform theory of Atkin--Lehner in the case of Hilbert modular forms: see e.g.\ Roberts--Schmidt \cite[p.~4]{RobertsSchmidt} for a concise statement.
\end{proof}

\begin{cor}
Let $\delta$ be the sign vector with $\delta_{\frakp}=-1$ if and only if $\frakp \mid \frakD$.  Then we have
\begin{equation} \label{eqn:Skimiseps}
   S_k(\widehat{\calO}^\times)^{\varepsilon} \simeq
    \bigoplus_{\frakM \mid \frakN}
    \bigoplus_{\varepsilon'} \bigl(S_k^{\rm new}(\widehat{\Gamma}_0(\frakM))^{\delta\varepsilon\varepsilon'}\bigr)^{\oplus \alpha^{\delta\varepsilon\varepsilon'}(\frakN/\frakM)}. 
    \end{equation}
\end{cor}

\begin{proof}
Repeat the proof of \Cref{thm:skOtimes} but restrict to the Atkin--Lehner subspaces using \Cref{prop:ALSK} and \Cref{prop:alltheAL}(b).
\end{proof}

\begin{remark}
It is also enough in \eqref{eqn:Skimis} and \eqref{eqn:Skimiseps} to sum only over those sign vectors $\varepsilon'$ such that $\varepsilon'_\frakp=1$ for all $\frakp \mid \frakN_-$.
\end{remark}

\section{Ternary quadratic forms and the even Clifford algebra} \label{sec:ternarycliff}

In this section, we review the even Clifford functor, relating ternary quadratic forms and quaternion orders.  We then observe that this functor gives a Hecke-equivariant isomorphism between certain spaces of orthogonal modular forms and quaternionic modular forms.  Importantly, we show that (because we are in odd rank) we can work equivalently with isometry classes or (twisted) similarity classes, the former being simpler for computation.  
We then restrict this result to a nice class of lattices---those with cyclic discriminant---corresponding to a nice class of quaternion orders---those which are locally residually unramified, as in the previous section.  For further background reading, see Voight \cite[Chapters 22--24]{Voight:quatbook}.

Let $R$ be a Dedekind domain with field of fractions $F \colonequals \Frac R$.  Let $V$ be an $F$-vector space with $\dim_F V =3$ and let $Q \colon V \to F$ be a nondegenerate quadratic form.  Recall that the \defi{similarity group} of $V$ is the algebraic group $\GO(V)$ consisting of linear maps that preserve the quadratic form $V$ up to a scalar called the \defi{similitude factor}, the \defi{isometry group} $\OO(V)$ of $V$ is the algebraic subgroup of similarities with similitude factor $1$ \cite[\S 4.2]{Voight:quatbook}, and the \defi{special} (or \defi{oriented}) \defi{isometry group} $\SO(V)$ of $V$ is the subgroup fixing an orientation.  

In fact, every similarity is in fact an oriented isometry composed with a scaling.

\begin{lem} \label{lem:SOV}
Multiplication gives an isomorphism
\begin{equation} \label{eqn:GOSO}
\SO(V) \times \GL_{1,F} \to \GO(V)
\end{equation}
of algebraic groups over $F$.  
\end{lem}

\begin{proof}
Choose a basis $V \simeq F^3$ and let $[T]$ be the Gram matrix of the symmetric bilinear form $T$ associated to $Q$ in this basis and $A=[\phi]$.  Certainly we have a homomorphism, it is injective because if $cA=1$ then $c^3\det A = c^3=1$ but also $A^{\textsf{t}} [T] A = [T]$ so $c^2=1$ so in fact $c = 1$.  

To see the map is surjective, we define an inverse.  If $A \in \GL_3(F)$ has $A^{\textsf{t}} [T] A = u[T]$ with $u \in F^\times$, let $d \colonequals \det A$, then taking determinants and cancelling $\det [T] \in F^\times$ gives $d^2=u^3$ so $u=(d/u)^2$ is a square; whence $(u/d)A \in \SO(V)$ and so we can map $A \mapsto ((u/d)A,d/u)$.  

This shows the isomorphism on $F$-points; repeating this over an arbitrary extension gives the result on the level of algebraic groups.  
\end{proof}

\begin{remark}
\Cref{lem:SOV} generalizes to nondegenerate quadratic forms of any odd dimension.
\end{remark}

Let $\Lambda \subseteq V$ be an $R$-lattice.  An \defi{isometry} or \defi{similarity} from $\Lambda$ to an $R$-lattice $\Lambda'$ is such a map $\phi$ on $V$ such that $\phi(\Lambda)=\Lambda'$, and we define $\SO(\Lambda) \leq \OO(\Lambda) \leq \GO(\Lambda)$ for the corresponding stabilizers.  Restricting \eqref{lem:SOV} gives an isomorphism
\begin{equation} \SO(\Lambda) \times \GL_{1,R} \xrightarrow{\sim} \GO(\Lambda). \end{equation}

Attached to $\Lambda$ will be four possible class sets corresponding to the notions above---but with one asterisk, they all end up being in natural bijection.  

We first define the usual \defi{genus} $\Gen_{\OO}(\Lambda) = \Gen(\Lambda)$ of $\Lambda$, consisting of all lattices $\Lambda' \subseteq V$ such that $\Lambda_{(\frakp)}$ is isometric to $\Lambda'_{(\frakp)}$ for all (nonzero) primes $\frakp \subseteq R$, where the subscript denotes localization; then we define the \defi{class set} $\Cl_{\OO} \Lambda = \Cl \Lambda$ to be the (global) isometry classes in $\Gen(\Lambda)$.  

Second, we could replace isometry with oriented isometry to get $\Cl_{\SO} \Lambda$, in fact the natural inclusion map $\Cl_{\SO} \Lambda \to \Cl_{\OO} \Lambda$ is a bijection: we have $\OO(\Lambda) \simeq \{\pm 1 \} \times \SO(\Lambda)$, so two lattices are isometric if and only if they are oriented isometric, the same being true locally.  (This holds more generally in odd rank.)  

Third, we replace isometry with similarity, and get the \defi{similarity genus} $\Gen_{\GO}(\Lambda)$ (lattices locally similar to $\Lambda$) and \defi{similarity class set} $\Cl_{\GO} \Lambda$ (up to global similarity).  On $\Cl_{\GO} \Lambda$, we have an action of $\Cl R$ by $[\fraka] \cdot [\Lambda] = [\fraka \Lambda]$.  

\begin{defn}
The \defi{twist} of $\Lambda$ by a fractional ideal $\mathfrak{a}$ is the lattice $\mathfrak{a} \Lambda$.  
A \defi{twisted similarity} between two $R$-lattices $\Lambda,\Lambda' \subseteq V$ is a similarity between a twist of $\Lambda$ and $\Lambda'$.  
\end{defn}

Finally, we define analogously the \defi{twisted similarity class set}, the quotient set $(\Cl_{\GO} \Lambda)^{\Cl R}$.  

\begin{lem} \label{lem:OOGO}
There is a natural inclusion map $\Cl_{\OO} \Lambda \hookrightarrow \Cl_{\GO} \Lambda$ which induces a bijection $\Cl_{\OO} \Lambda \leftrightarrow (\Cl_{\GO} \Lambda)^{\Cl R}$.  
\end{lem}

\begin{proof}
The two maps are injective, since isometries are in particular similarities and twisted similarities.  The second map is also surjective, as follows.  Suppose $\Lambda'$ is locally twisted similar to $\Lambda$; we find a twist of $\Lambda'$ which is locally isometric to $\Lambda$.  Since the localizations are DVRs we conclude that $\Lambda'$ is locally similar to $\Lambda$.  Considering a local similarity over $F$, using \eqref{eqn:GOSO} we see that there exists local isometries between $\Lambda_{(\frakp)}$ and $a_{(\frakp)} \Lambda'$ for each prime $\frakp$.  Since $\Lambda_{(\frakp)}=\Lambda'_{(\frakp)}$ for all but finitely many $\frakp$, there is a unique fractional ideal $\fraka \subseteq F$ such that $\fraka_{(\frakp)}=a_{(\frakp)} R_{(\frakp)}$ for all $\frakp$, and hence $\Lambda$ is locally isometric to $\fraka\Lambda'$, as desired.
\end{proof}

We now turn to the main goal of this section: Clifford algebras.  The even Clifford algebra \cite[\S 5.3]{Voight:quatbook} $B \colonequals \Clf^0(Q)$ is a quaternion algebra over $F$ (defined by a universal property \cite[Exercise 5.20]{Voight:quatbook}); and we recover $Q=\nrd|_{B^0}$ up to similarity as the reduced norm on the trace zero subspace.  The construction of the even algebra extends naturally to $R$-lattices in $V$.  

\begin{thm} \label{thm:functorial}
The association $\Lambda \mapsto \Clf^0(\Lambda)$ of the even Clifford algebra to a lattice defines a functor from the category of
\begin{center}
\emph{$R$-lattices $\Lambda \subseteq V$, under similarities}
\end{center}
to the category of
\begin{center}
\emph{Gorenstein quaternion $R$-orders $\calO \subseteq B$, under isomorphisms}
\end{center}
Moreover, the association is functorial with respect to ring homomorphisms $R \to R'$.
\end{thm}

% By the latter functoriality statement, we mean that formation of the even Clifford algebra commutes with tensor product (over a ring homomorphism): given a ring homomorphism $R \to R'$, the canonical isomorphism
% \[ \Clf_{R}^0(\Lambda) \otimes_R R' \simeq \Clf_{R'}^0(\Lambda \otimes_R R') \]
% (labeling the Clifford maps over $R$ and $R'$) is natural with respect to the given morphisms.  

\begin{proof}
See Voight \cite[Theorems 22.2.11 and 22.3.1]{Voight:quatbook}.  Restricting $R$ to a Dedekind domain implies that every order is projective over $R$, and we obtain only Gorenstein orders \cite[Theorem 24.2.10]{Voight:quatbook} because our quadratic forms are necessarily primitive: we take the codomain of the quadratic module to be the $R$-submodule generated by the values of the quadratic form.  
\end{proof}

\begin{cor} \label{thm:evencliffbij}
The even Clifford map gives a bijection between the set of twisted similarity classes of $R$-lattices $\Lambda \subseteq V$ and the set of isomorphism classes of Gorenstein $R$-orders $\calO \subseteq B$.  Moreover, it induces a natural bijections 
\[ \Cl_{\GO}(\Lambda)^{\Cl R} \leftrightarrow \Typ \calO \] 
to the type set of $\calO=\Clf^0(\Lambda)$.  
\end{cor}

By \Cref{lem:OOGO}, we similarly get a natural bijection 
\begin{equation} \label{eqn:yupClO}
\Cl_{\SO}(\Lambda) \leftrightarrow \Typ \calO.
\end{equation}

\begin{proof}
For the first statement, apply Voight \cite[Main Theorem 22.5.7]{Voight:quatbook} with the same modifications as in the proof of \Cref{thm:functorial}; an explicit inverse is also given (arising from the reduced norm).  The second statement is Voight \cite[Corollary 22.5.12]{Voight:quatbook}.
\end{proof}

\begin{remark}
Although it is not needed in what follows, the bijections in \Cref{thm:evencliffbij} also follow naturally from an equivalence of categories, something which holds over a more general base ring: this is developed in Voight \cite[Theorem B]{Voight:char}.  Algorithmically it is most convenient to work with \eqref{eqn:yupClO} as isometries are easier to compute, so we explain this associated equivalence of categories in \Cref{thm:invassocfunc}.
\end{remark}

We now apply this to orthogonal modular forms.  Recall \cite[Proposition 4.5.10]{Voight:quatbook} that there is an exact sequence
\begin{equation} \label{eqn:blang}
1 \to F^\times \to B^\times \to \SO(\nrd|_{B^0}) \to 1 
\end{equation}
giving an isomorphism $B^\times/F^\times \simeq \SO(Q)(F)$ (and $Q \simeq \nrd|_{B^0}$).  This extends integrally: for a lattice $\Lambda \subseteq V$, computing stabilizers gives 
\begin{equation} 
\SO(\Lambda)(R) \simeq N_{B^\times}(\calO)/F^\times 
\end{equation}
with $\calO \colonequals \Clf^0(\Lambda)$.

Let $\rho \colon \SO(V) \to \GL(W)$ be a finite-dimensional representation.  Via \eqref{eqn:blang} we can lift to get a representation $\rho \colon B^\times \to \GL(W)$ whose restriction to $F^\times$ is trivial, i.e., has trivial central character.  

% Abbreviate
% \begin{equation}
% \AL(\calOhat) \colonequals N_{\widehat{B}^\times}(\widehat{O})/\widehat{F}^\times \widehat{\calO}^\times.
% \end{equation}

Recalling the Atkin--Lehner group $\AL(\calOhat)$ in \eqref{eqn:ALOHat}, we have the following result.

\begin{thm} \label{thm:Heckeequiv}
The even Clifford map induces a Hecke-equivariant bijection
\[ \sM_W(\SO(\widehat{\Lambda})) \xrightarrow{\sim}  \sM_W(\widehat{\calO}^\times)^{\AL(\calOhat)} \]
that restricts to a Hecke-equivariant bijection cuspidal subspaces
\[ S_W(\SO(\widehat{\Lambda})) \xrightarrow{\sim} S_W(\widehat{\calO}^\times)^{\AL(\calOhat)}. \]
\end{thm}

\begin{proof}
From \Cref{thm:evencliffbij}, we have a bijection $\Cl \Lambda \leftrightarrow \Typ \calO$; from \Cref{thm:functorial} (restricting to oriented isometries), it is functorial with respect to the Hecke operators.  The constant functions and inner product are preserved by this map, so the result restricts to their orthogonal complements, the cuspidal subspaces.
% Our work is in giving a Hecke-equivariant bijection from maps on the type set to $W$ and 
% \[ \sM(N_{\widehat{B}^\times}(\widehat{O}),W). \]
% Each $\widehat{\alpha} \in N_{\widehat{B}^\times}(\widehat{O})$ defines an operator on $\sM(\widehat{O}^\times,W)$ by multiplication on the right by $\widehat{\alpha}$, this action preserves the map to the type set, it acts transitively on the fibers, and it is Hecke equivariant. 
\end{proof}

\Cref{thm:Heckeequiv} applies to a general ternary lattice with arbitrary weight, yielding a general (Gorenstein) quaternion order.  Looking towards our application, where we restrict our classes of orders to those which are residually unramified, we characterize the corresponding lattices.  

\begin{lem} \label{lem:equivconds}
Let $\Lambda$ be a ternary $R$-lattice.  Then the following are equivalent.
\begin{enumroman}
\item The discriminant module $\Lambda^\sharp/\Lambda$ is a cyclic $R$-module;
\item For all (nonzero) primes $\frakp \subseteq R$ we have 
\begin{equation} \label{eqn:psplitting}
\Lambda_\frakp \simeq U_\frakp \boxplus R_\frakp z
\end{equation}
where $U_\frakp$ is nonsingular (i.e., $\disc(U_\frakp) \in R_\frakp^\times/R_\frakp^{\times 2}$); and
\item $\Clf^0(\Lambda)$ is residually unramified at all primes $\frakp$.  
\end{enumroman}
\end{lem}

\begin{proof}
The equivalences of (i) and (ii) are almost immediate: for odd primes $\frakp$, the condition in (ii) is equivalent to the rank of the reduction of $Q$ modulo $\frakp$ (defined on $\Lambda/\frakp\Lambda$) being at least $2$, but when $\frakp$ is even we are asking further that the reduction is nonsingular.  The equivalence (ii) $\Leftrightarrow$ (iii) is given in Voight \cite[24.3.9]{Voight:quatbook}.
\end{proof}

\section{The radical character} \label{sec:radchar}

% This is Lemma \ref{lem:equivconds}
% {\color{cyan}
% \begin{prop}[Will be done in sect 4]
% The order $\calO$ is residually unramified if and only if
% the discriminant $\Lambda^\#/\Lambda$ is cyclic as an $R$-module.
% \end{prop}
% \begin{proof}
% See [DPRT, sect 5], where these are called special lattices.
% \end{proof}
% }

We keep notations as in previous section: in particular, let $\Lambda \subset V$ be a definite ternary $R$-lattice of discriminant $\frakN$, and let $\calO = \Clf^0(\Lambda)$ be its even Clifford algebra, an  $R$-order in $B=\Clf^0(V)$.  

Let $\frakp\mid\frakN$.  We define the \defi{radical} of $\Lambda_\frakp$ as
\begin{equation}
\Rad(\Lambda_\frakp) \colonequals \{x\in\Lambda_\frakp : T(x,y) \equiv 0\psmod{2\frakN}\text{ for all $y \in \Lambda_\frakp$}\}.
\end{equation}
The name is justified, as this recovers the usual radical of a quadratic space when the discriminant is $0$.  

% \begin{rmk}
% For odd $\frakp$, the condition is equivalent to $T(v,w) \rangle \equiv 0 \pmod{\frakp}$, and for $\frakp^e \parallel 2R$, this is $T(v,w) \equiv 0 \pmod{\frakp^{e+1}}$; we write these uniformly as above.  
% \end{rmk}

We see that $\Rad_\frakp(\Lambda)$ is a $R_\frakp$-lattice with
$\Rad(\Lambda_\frakp)\supseteq 2\frakN\Lambda_\frakp$.  

\begin{prop} \label{prop:cycooRad}
If $\Lambda_\frakp$ has cyclic discriminant, then $\Rad(\Lambda_\frakp)/(2\frakN\Lambda_\frakp)$ is free of rank 1 over $R_\frakp/2\frakN R_\frakp$.
\end{prop}

\begin{proof}
We apply \Cref{lem:equivconds}, (i) $\Rightarrow$ (ii): it is generated by $z$, since $Q(z) \equiv 0 \pmod{\frakN}$ and hence $T(z,z)=2Q(z) \equiv 0 \pmod{2\frakN}$; see also \cite[Lemma 6.13]{DPRT}.
\end{proof}

Since it is defined by the bilinear form, $\Rad(\Lambda_\frakp)$ is invariant
under $\OO(\Lambda_\frakp)$, hence we obtain a character
\begin{equation} \label{eqn:radicalchar}
    \nu_\frakp \colon \SO(\Lambdahat)\to \SO(\Lambda_\frakp) \to \OO(\Rad_\frakp(\Lambda)/2\frakN \Lambda_\frakp) \simeq (R_\frakp/2\frakN R_\frakp)^\times
\end{equation}
which we call the \defi{radical character} at $\frakp$.

\begin{lem}
The image of the radical character $\nu_\frakp$ is $\{\pm 1\}$, and it is uniquely defined by
\[
  \sigma(z) \equiv \nu_\frakp(\sigma)z \pmod{2\frakp\Lambda}
\]
for $\sigma \in \SO(\Lambda_\frakp)$.
\end{lem}

\begin{proof}
For the first statement, we have a (noncanonical) splitting $\Lambda_\frakp \simeq U_\frakp \boxplus R_\frakp z$ in \Cref{lem:equivconds}(ii) with $z \in \Rad(\Lambda_\frakp)$.  Let $\sigma \in \SO(\Lambda_\frakp)$.  Then from \eqref{eqn:radicalchar} we have $\sigma(z) = az+y$ with $a \in R_\frakp^\times$ and $y \in 2\frakN \Lambda_\frakp$.  Without loss of generality, we may suppose that $y \in 2\frakN U_\frakp$.  Then $Q(z)=Q(\sigma z) = a^2Q(z) + Q(y) \equiv a^2Q(z) \pmod{4\frakN^2}$ since $y \in 2\frakN U_\frakp$.  Since $\ord_\frakp Q(z) = \ord_\frakp \frakN$, we conclude that $a^2 \equiv 1 \pmod{4\frakN}$, whence $a=\nu_\frakp(\sigma) \equiv \pm 1 \pmod{2\frakN}$.  Composing a reflection in $z$ with reflection in a vector in $U_\frakp$ shows that the image is fully $\{\pm 1\}$, and this can already be seen modulo $2\frakp$.
\end{proof}

We extend the radical character multiplicatively, defining 
\begin{equation}
\nu_\frakM \colonequals \prod_{\frakp \mid \frakM} \nu_\frakp
: \SO(\Lambdahat)\to \{\pm 1\}
\end{equation}
for $\frakM \mid \frakN$ squarefree.
%A more practical description is as follows. A \emph{$\frakp$-radical vector} is
%a generator of $\Rad_\frakp(\Lambda)$.

There is another character, the spinor norm character; we show these agree in odd exponent, as follows.  Recall that the \defi{spinor norm} is the map composition 
\begin{equation}
\theta \colon \SO(V) \simeq B^\times/F^\times \xrightarrow{\nrd} F^\times/F^{\times 2}.
\end{equation}
We define the \defi{spinor norm character} for $\frakM \mid \frakN$ by composing
\begin{equation}  \label{eqn:sigmaMM}
\begin{aligned}
\theta_\frakM \colon \SO(V) \xrightarrow{\theta} F^\times/F^{\times 2} &\to \{\pm 1\} \\
a &\mapsto \prod_{\frakp \mid \frakM} (-1)^{\ord_\frakp(a)}.
\end{aligned}
\end{equation}

We let $A \colonequals \langle \nu_{\frakM} \rangle \simeq \prod_{\frakp \mid \frakN} C_2$ be the \defi{radical character group}.  

\begin{prop} \label{prop:A21}
Suppose $\calO$ is residually unramified.  Then the following statements hold, under the isomorphism $\SO(\Lambdahat) \xrightarrow{\sim} N_{\Bhat^\times}(\calOhat)$ induced by the even Clifford map.  
\begin{enumalph}
\item The pairing 
\begin{equation}
\begin{aligned}
A \times N_{\Bhat^\times}(\calOhat) &\to \{\pm 1\} \\
(\nu,\alpha) &\mapsto \nu(\alpha)
\end{aligned}
\end{equation}
is perfect.  
\item If $\ord_\frakp \frakN$ is \emph{odd}, then $\nu_\frakp = \theta_\frakp$.  
\end{enumalph}
\end{prop}

\begin{proof}
First part (a).  We computed $N_{B_\frakp^\times}(\calO_\frakp) = F_\frakp^\times \calO_\frakp^\times \langle w_\frakp \rangle$ in \Cref{prop:alltheAL}(i).  We now make that a bit more explicit.  Writing
\[ \Lambda_\frakp = U_\frakp \boxplus R_\frakp z \]
we compute that
\[ \Clf^0(\Lambda_\frakp) = \Clf^0(U_\frakp) \oplus U_\frakp z. \]
Write $S_\frakp \colonequals \Clf^0(U_\frakp)$.  Then $S_\frakp$ is unramified over $R_\frakp$, either split or inert according as $U_\frakp$ is isotropic or anisotropic and according as $\frakp \mid \frakN_{+}$ or $\frakp \mid \frakN_{-}$.  Abbreviate $J_\frakp \colonequals U_\frakp z$.  Then $J_\frakp = S_\frakp j$ with $j^2=b$ with $\ord_\frakp(b)=\ord_\frakp \frakN=e$, and $j=w_\frakp$.  

Suppose first that $\alpha_\frakp \in S_\frakp^\times$.  Then $\alpha_\frakp$ commutes with $z$, since $U_\frakp$ is orthogonal to $z$.  Similarly, suppose $\alpha_\frakp = 1 + \beta_\frakp \in 1 + J_\frakp$.  Then $\Nm(\alpha_\frakp)=1-\beta_\frakp^2$ with $\beta_\frakp^2 \in \frakN$, and $\beta_\frakp z \in \frakN \Lambda_\frakp$.  Hence
\begin{equation}
\begin{aligned}
\alpha_\frakp^{-1} z \alpha_\frakp &= (1+\beta_\frakp)^{-1}(1-\beta_\frakp)z = \frac{(1-\beta_\frakp)^2}{1-\beta_\frakp^2}z \\
&\equiv \frac{1+\beta_\frakp^2}{1-\beta_\frakp^2} z - \frac{2\beta_\frakp}{1-\beta_\frakp^2}z \equiv z \pmod{2\frakN\Lambda_\frakp} 
\end{aligned}
\end{equation}
since $\beta_\frakp^2 \equiv -\beta_\frakp^2 \pmod{2\frakN}$.  This shows that $\nu_\frakp$ is trivial on $\calO_\frakp^\times = S^\times \cdot (1+J_\frakp)$.  Thus it must be nontrivial on $j$, and indeed we quickly verify that $j^{-1} z j = j^{-1} (-j) z = -z$.  

This shows that the radical character on the normalizer is simply the projection onto 
\begin{equation} \label{eqn:NBw}
N_{B_\frakp^\times}(\calO_\frakp)/F_\frakp^\times \to N_{B_\frakp^\times}(\calO_\frakp)/(F_\frakp^\times \calO_\frakp^\times) \simeq \langle w_\frakp \rangle \simeq \{\pm 1\}.
\end{equation}  
Part (a) follows directly.  

For part (b), since $\ord_\frakp \frakN=\ord_\frakp(b)=\nrd(j)$ is odd, the map in \eqref{eqn:NBw} is also exactly described by the spinor norm character $\theta_\frakp$ in \eqref{eqn:sigmaMM}.  (And this does not hold when $\ord_\frakp \frakN$ is even, in that case $\theta_\frakp=1$)  
\end{proof}

\begin{cor} \label{cor:ALinterp}
Let $\frakM \mid \frakN$ be squarefree, and let $\varepsilon_\frakM$ be the sign vector with $(\varepsilon_\frakM)_\frakp = -1$ if and only if $\frakp \mid \frakM$.  Suppose that $\calO$ is residually unramified.  Then there is a Hecke-equivariant bijection 
\begin{equation}  \label{eqn:MWSW}
M_k(\SO(\widehat{\Lambda}),\nu_\frakM) \xrightarrow{\sim} M_k(\widehat{\calO}^\times)^{\varepsilon_\frakM}
\end{equation}
where the image is the space of forms where the Atkin--Lehner operators have signs agreeing with $\varepsilon_\frakM$.  The map \eqref{eqn:MWSW} restricts to a similar bijection on cuspidal subspaces.
\end{cor}

\begin{proof}
We apply \Cref{thm:Heckeequiv} and identify the fixed subspace for the characters using \Cref{prop:A21}.
\end{proof}

\section{Newform theory for orthogonal modular forms} \label{sec:newforms}

In this section, we translate the degeneracy maps and newform theory provided in \cref{sec:hmf} to the orthogonal side.  This upgrades \Cref{cor:ALinterp} to a statement about newforms, and it explains the multiplicities of oldforms observed over $F=\Q$ by Birch \cite[pp.\ 202--203]{Birch}.

The degeneracy maps on the quaternion (and Hilbert) side are of the following form.  Whenever we have $\calO \subseteq \calO'$, equivalently $\calOhat \subseteq \calOhat'$, we get a natural surjective map of pointed sets 
\begin{equation}
\begin{aligned}
\xymatrix{ 
B^\times \backslash \widehat{B}^\times / N_{\widehat{B}^\times}(\calOhat) = \Typ(\calO) \ar[d] \\
B^\times \backslash \widehat{B}^\times / N_{\widehat{B}^\times}(\calOhat') = \Typ(\calO').
}
\end{aligned}
\end{equation}
Described in terms of type sets, given a order locally isomorphic to $\calO$, its adelization is $\alphahat^{-1} \calOhat \alphahat$; the map then makes the order $\alphahat^{-1} \calOhat' \alphahat$ which when intersected with $B$ gives an order locally isomorphic to $\calO'$.  

But we have already identified these double cosets (and type sets) using the even Clifford functor \eqref{eqn:yupClO}: corresponding to $\calOhat'$ is an adelic lattice $\Lambdahat' \subseteq \widehat{V}$ and lattice $\Lambda' = \Lambdahat' \cap V \subseteq V$, giving again a map of pointed sets
\begin{equation}
\begin{aligned}
\xymatrix{ 
\Cl(\Lambda) = \SO(V) \backslash \SO(\widehat{V}) / \SO(\widehat{\Lambda}) \ar@{<->}[r] \ar[d] & 
B^\times \backslash \widehat{B}^\times / N_{\widehat{B}^\times}(\calOhat) = \Typ(\calO) \ar[d] \\
\Cl(\Lambda') = \SO(V) \backslash \SO(\widehat{V}) / \SO(\widehat{\Lambda}') \ar@{<->}[r] & 
B^\times \backslash \widehat{B}^\times / N_{\widehat{B}^\times}(\calOhat') = \Typ(\calO')
}
\end{aligned}
\end{equation}
(We do not claim that $\Lambdahat \subseteq \Lambdahat'$, only that $\SO(\Lambdahat) \leq \SO(\Lambdahat')$.)  On the level of class sets, the map is described in a similar way to the quaternionic case: for a lattice which is locally isometric to $\Lambda$ its adelization is $\widehat{g} \Lambdahat$, we make the adelic lattice $\widehat{g} \Lambdahat'$ and then intersect with $V$ to get a lattice locally isometric to $\Lambda'$. 

Orthogonal and quaternionic modular forms are defined by equivariant maps on these sets.  The degeneracy operators on quaternionic modular forms commute with   
\begin{equation} 
\alpha_{\Lambda'} \colon M_W(\SO(\Lambdahat')) \hookrightarrow M_{W'}(\SO(\Lambdahat)) 
\end{equation}
via pullback.  These maps give an old and new subspace as usual.

%\begin{lem}
%For the representation $W=W_k \otimes \nu_{\frakM}$, we may take $W'=W_k \otimes \nu_{\frakM'}$ where $\frakM'$ removes level according to the index of $\calO$ in $\calO'$.
%\end{lem}
%
%\begin{proof}
%Something to check.  \jv{We really should be able to get something to work with the radical character and Atkin--Lehner involution, but there may be a change of sign?}
%\end{proof}

\begin{defn}
The \defi{old subspace} $M_W^{\textup{old}}(\SO(\Lambdahat)) \subseteq M_W(\SO(\Lambdahat))$ is the span of the image of all degeneracy maps $\alpha_{\Lambda'}$ for lattices $\Lambda'$ corresponding to superorders $\calO' \supseteq \calO$.  The \defi{new subspace} $M_W^{\textup{new}}(\SO(\Lambdahat)) \subseteq M_W(\SO(\Lambdahat))$ is the orthogonal complement of the old subspace with respect to the Petersson inner product.
\end{defn}

%\begin{rmk}
%We are in a lucky situation in the compact setting that the Petersson inner product extends to the entire space of modular forms (no convergence issues).  Still our main focus is cusp forms.
%\end{rmk}
%
%\begin{lem} \label{lem:orthsame}
%The following statements hold.
%\begin{enumalph}
%\item The orthogonal degeneracy maps commute with Hecke operators (away from the level).  
%\item Under the even Clifford map, the orthogonal degeneracy maps correspond to quaternionic degeneracy maps. \end{enumalph}
%\end{lem}
%
%\begin{proof}
%Part (a) follows from the fact that they are locally defined at $\frakp \mid \frakN$ away from the Hecke operators with $\frakp \nmid \frakN$.  
%
%Part (b) is a local calculation which can be done over a field, it breaks down into cases depending on if the quaternion algebra $B$ is split or not.   
%\end{proof}

\begin{thm}  \label{thm:mainthmo}
If $\calO$ is residually unramified, then there is a Hecke-equivariant bijection 
\[ M_k^{\textup{new}}(\SO(\Lambdahat),\nu_\frakM) \xrightarrow{\sim} M_k^{\textup{new}}(\calOhat)^{\eps_\frakM} \]
that restricts to the cuspidal subspaces.
\end{thm}

\begin{proof}
We restrict \Cref{cor:ALinterp} to the new subspace on quaternionic modular forms; by definition, this corresponds to the new subspace on orthogonal modular forms.
%The old subspaces on both sides of the Hecke equivariant isomorphism \Cref{cor:ALinterp} match up by \Cref{lem:orthsame}.  To get them to exactly match up, we either need to know that the map also respects orthogonality in the Petersson inner product (even though it does not literally preserve the inner product), or we need to appeal to multiplicity one on the new subspace (on the quaternionic side is enough).   
\end{proof}

%%%%%%%%%%%%%%%%%%%%%%%%%%%%%%%%%%%%%%%%%%%%%%%%%%%%%%%%%%%%%%%%%%%%%%%%

\section{Orthogonal algorithms} \label{sec:orthalg}

A general reference for algorithms on orthogonal modular forms on lattices is Greenberg--Voight \cite{gv}.  Further practical improvements are given by Assaf--Fretwell--Ingalls--Logan--Secord--Voight \cite[\S 3]{defortho}.

We advance these algorithms with the following additions.  First, we construct a ternary quadratic space with cyclic discriminant module given invariants, as follows.

\begin{alg} \label{alg:ternary}
Given as input
\begin{itemize}
\item a totally real field $F$, 
\item a factored ideal $\frakN=\prod_{i=1}^r \frakp_i^{e_i}$ of $R=\Z_F$, 
\item a divisor $\frakN_{-} \mid \frakN$ such that if $\frakp^e \parallel \frakN_{-}$ then $e$ is odd, and
\item a squarefree divisor $\frakD \mid \frakN_{-}$ whose number of prime divisors has the same parity as $[F:\Q]$;
\end{itemize}
we return as output a residually unramified order $\calO$ in a suitable quaternion algebra of discriminant $\frakD$.    
\begin{enumalg}
\item Compute a totally definite quaternion algebra $B$ of discriminant $\frakD$ \cite[Algorithm 4.1]{gv}.
\item Compute a maximal order $\calO \subseteq B$ \cite[\S 7]{Voight:matring}.
\item Compute an Eichler order $\calO_1 \subseteq \calO$ of level $\frakN_+ \colonequals \frakN/\frakN_{-}$ by computing a $\frakp$-matrix ring for all $\frakp \mid \frakN_+$ \cite[Corollary 7.13]{Voight:matring} and taking the preimage of the standard Eichler order.
\item Compute $\calO_2 \subseteq \calO_1$ residually inert for $\frakp \mid \frakD$, for each prime $\frakp$:
\begin{enumalgalph}
\item Compute $J_\frakp \colonequals \rad(\calO_\frakp) = [\calO_\frakp,\calO_\frakp]$.
\item Lift a basis of the quotient $\calO_\frakp/J_\frakp$ to $S_\frakp$.
\item Take $\calO_{2,\frakp} = S_\frakp + J_\frakp^{(e+1)/2}$ for $\frakp^e \parallel \frakN_{-}$ (with $e$ odd).
\end{enumalgalph}
\item Compute $\calO_3 \subseteq \calO_2$ residually inert for $\frakp \mid \frakN_{-}$ but $\frakp \nmid \frakD$:
\begin{enumalgalph}
\item Compute a $\frakp$-matrix ring $\calO_\frakp \simeq \M_2(R_\frakp)$. 
\item Compute $S_\frakp$, the valuation ring in the quadratic unramified extension $K_\frakp \supseteq F_\frakp$.  
\item By the regular representation, compute $S_\frakp \hookrightarrow \M_2(R_\frakp)$ and identify $S_\frakp$ with its preimage in $\calO_\frakp$.  Compute $\calO_{\frakp}' = S_\frakp + \frakp \calO_\frakp$, let $J_\frakp = \rad(\calO_\frakp')$, and take $\calO_{3,\frakp}' = S_\frakp + J_\frakp^{e/2}$.
\end{enumalgalph}
\item Return $\calO_3$.
\end{enumalg}
\end{alg}

\begin{thm}
\Cref{alg:ternary} returns correct output and runs in probabilistic polynomial time.
\end{thm}

\begin{proof}
Steps 1--3 run in probabilistic polynomial time by the references given.  (See in particular Voight \cite{Voight:matring} for a discussion of exact global representations for $\frakp$-adic computations.)  In steps 4 and 5 we have only linear algebra steps other than finding an irreducible quadratic polynomial over $\F_\frakp$, which can be found in probabilistic polynomial time. 
\end{proof}

To compute spaces of orthogonal modular forms with radical character, there is no trouble computing the radical character: part of the above suite of algorithms efficiently computes a $\frakp$-splitting as in \eqref{eqn:psplitting}, and from there we just compute the action on the radical generator $z$ (by matrix multiplication).  

In the special case where the radical character is given spinor norm (when $\ord_\frakp \frakN$ is odd, \Cref{cor:ALinterp}), we have the following extraordinary formula.

\begin{lemma} \label{lem:spinornorm}
Let $V$ be a finite-dimensional nondegenerate quadratic space over a field $k$ with $\opchar k \neq 2$.  Let $\sigma \in \SO(V)$ have $\det(1+\sigma) \neq 0$.  Then the spinor norm of $\sigma$ is
\[ \theta(\sigma)=\det((1+\sigma)/2) = 2^n \det(1+\sigma) \in k^{\times}/k^{\times 2}. \]
If further $n=3$, then $\theta(\sigma)=1+\tr(\sigma)$.  
\end{lemma}

\begin{proof}
See Zassenhaus \cite[Theorem, Corollary, p.~446]{Zassenhaus}, who actually shows one you can \emph{define} the spinor norm this way (without the Clifford algebra).  For completeness, we give a direct proof in the case $n=3$.  

    The case $\sigma=\id_V$ is immediate.  By the Cartan--Dieudonn\'e theorem, 
    $\sigma=\tau_u\tau_v$ is a product of reflections where $u,v \in V$ are not colinear, where
    \[
        \tau_u(x) = x - \frac{T(x,u)}{Q(u)}u
    \]
    (and similarly for $v$).  Let $w\in V$ be orthogonal to
    $u$ and $v$, so that $\set{u,v,w}$ is a basis of $V$.

    We compute $\sigma$ in this basis:
    \begin{align*}
        \sigma(u) &= \tau_u\left(u-\frac{T(u,v)}{Q(v)}v\right)
        = -u - \frac{T(u,v)}{Q(v)}\left(v-\frac{T(v,u)}{Q(u)}u\right) \\
        & = \left(\frac{T(u,v)^2}{Q(u)Q(v)}-1\right)u
          - \frac{T(u,v)}{Q(v)}v
        \\
        \sigma(v) &= \tau_u(-v) = \frac{T(v,u)}{Q(u)}u -v
        \\
        \sigma(w) &= w
    \end{align*}
    Hence 
    \[ \det(1+\sigma) = \frac{T(u,v)^2}{Q(u)Q(v)} \]
    and
    \[ \tr(\sigma) = \frac{T(u,v)^2}{Q(u)Q(v)}-1, \] 
    and the
    claim follows since $\theta(\sigma)=Q(u)Q(v) \in k^\times/k^{\times 2}$.  
\end{proof}

For completeness, when $\tr(\sigma)=-1$, we can still do something nice.  

\begin{cor} \label{lem:trminus1}
If $\sigma \in \SO(V)$ with $\dim_k V =3$ has $\tr(\sigma)=-1$, then $\sigma^2=\id_V$ and $\ker(\sigma-1)=kw$ for some nonzero $w \in V$ and $\theta(\sigma)=Q(w)$.
\end{cor}

\begin{proof}
If $\tr(\sigma)=-1$ then $\sigma=\tau_u\cdot\tau_v$ with $T(u,v)=0$, so taking the orthogonal complement we find an orthogonal basis $\beta=\{u,v,w\}$ for $V$ such that $[\sigma]_\beta=\diag(-1,-1,1)$ is diagonal, and in particular $\sigma^2=1$ and $\ker(\sigma-1)=k w$.  Since the spinor norm of $\sigma$ is by definition $\theta(\sigma)=Q(u)Q(v) \in k^{\times}/k^{\times 2}$ and the discriminant of $Q$ is $\disc Q = 4 Q(u)Q(v)Q(w)$, the result holds.  
\end{proof}

To find a vector $w$ with $\sigma(w)=w$ as in \Cref{lem:trminus1}, better than computing a kernel, we see that since $\sigma$ is an involution, we can take $w \colonequals v+\sigma(v)$ for any $v \in V$ whenever $w \neq 0$.  Since $\sigma \neq -1$, we need only try $v$ in a basis for $V$.

We conclude with a runtime analysis, which we state in more generality.  

\begin{thm}
There exists an explicit algorithm that, given as input a totally positive definite ternary quadratic space $V$ over a totally real number field $F$ with $n=[F:\Q]$, an $R$-lattice $\Lambda \subset V$, a nonzero prime ideal $\frakp \subset R$, computes the Hecke operator $T_{\frakp}$ using 
\[ \widetilde{\mathcal{O}}(\Nm(\frakp) H_{3n}(\|\Lambda\| d^2)) \]
bit operations, where $\|\Lambda\|$ denotes the number of bits required to encode $\Lambda$, $d=\#\Cl \Lambda$, and $H_{3n}$ is a polynomial that determines the number of bit operations to compute a Hermite normal form of a matrix in $\M_{3n}(\Z)$.
\end{thm}

\begin{proof}
The algorithm is specified in \cite[\S 6]{gv}; we focus on the ternary case.  We first enumerate the points of the projective conic $Q_{\Lambda}(x,y,z) \equiv 0 \pmod{\frakp}$ corresponding to isotropic lines, which can done in time $\widetilde{\mathcal{O}}(\Nm(\frakp))$ by looping over values of $x \in \F_\frakp$ and computing roots.  In a fixed number of operations we can compute a hyperbolic complement, compute generators for the neighbor, and then we compute a Hermite normal form over $\Z$ which by assumption has complexity $\mathcal{O}(H_{3n}(\|\Lambda\|))$.  Finally, we account for the isometry testing for the lattice projected to $\Z$ which occurs at a complexity of $\mathcal{O}((3n)^{O(3n)})$ by work of Haviv--Regev \cite{HR}: we treat $V$ as a $\Q$-vector space with additional bilinear forms.  
\end{proof}

%%%%%%%%%%%%%%%%%%%%%%%%%%%%%%%%%%%%%%%%%%%%%%%%%%%%%%%%%%%%%%%%%%%%%%%%

\section{Implementation and examples} \label{sec:examples}

This algorithm has been implemented in Magma \cite{Magma} over a general totally real field (in arbitrary rank), as reported by Assaf--Fretwell--Ingalls--Logan--Secord--Voight \cite{defortho}.  

% \begin{exm}
% We begin by instantiating the newform theory and multiplicities discussed in \Cref{prop:ALSK}.  \jv{Add interesting example, maybe level $N_+=3^2$ and $N_-=2^3 5$ so $N=360$?}
% \end{exm}

We also implemented an optimized version restricted to $F=\Q$ and $N$ squarefree, written in C++ and available online \cite{code}.  On a standard desktop we can compute all rational newforms with squarefree conductor $N < 15\,000$ and their Hecke eigenvalues for all $p < 1000$ in just 25 minutes!

\begin{exm} \label{exm:blang}
For level $N=1\,062\,347=11\cdot 13\cdot 17 \cdot 19 \cdot 23=N_-$ so $N_+=1$ and hence we directly get the space of newforms, using \Cref{alg:ternary} we find the ternary form
\[ Q(x,y,z)=x^2+187y^2+1467z^2-187xz \]
with $\#\Cl(\Lambda)=2016$.  

Given $Q$, we can compute the Hecke matrices $[T_2],[T_3],[T_5],[T_7]$ for \emph{all} radical (spinor) characters, giving all newforms, in 4 seconds on a standard desktop machine.  Then $1$ minute of linear algebra computing kernels with sparse matrices in Magma following methods in Cremona \cite[\S 2.7]{Cremona} gives that there are exactly $5$ elliptic curves with conductor $N$.  

The same computation with modular symbols Magma crashed after consuming all 24 GB of available memory! 
\end{exm}

The excellent performance in \Cref{exm:blang} suggests that this method would be an effective method to get a ``random'' isogeny class of elliptic curves of moderately large conductor.  

%%%%%%%%%%%%%%%%%%%%%%%%%%%%%%%%%%%%%%%%%%%%%%%%%%%%%%%%%%%%%%%%%%%%%%%%

\appendix

\section{A functorial inverse to the even Clifford map} \label{AppendixA}

In \cref{sec:ternarycliff}, we saw that the even Clifford functor associates to a nondegenerate ternary quadratic module a quaternion $R$-order, giving a bijection from twisted isometry classes of lattices to isomorphism classes of orders and inducing a bijection on class sets.  In this appendix, we upgrade this to an equivalence of categories extending work of Voight \cite[Theorem B]{Voight:char}.

Throughout the appendix, let $R$ be a (commutative) domain with $F \colonequals \Frac R$.  (The argument generalize to an arbitrary base, but we keep this hypothesis in line with the paper and to simplify a few arguments.)  The even Clifford functor associates to a nondegenerate ternary quadratic module a quaternion $R$-order.  We show how to do the converse, furnishing a functorial inverse to the Clifford functor.  The construction is due to Voight \cite[\S 2]{Voight:char}, following Bhargava \cite{BhargavaQuartic} (who considered the case of commutative rings of rank $4$) and a footnote of Gross--Lucianovic \cite[Footnote 2]{GrossLucianovic}.

Let $\calO \subset B$ be an $R$-order in a quaternion algebra $B$ over $F$ which is projective as an $R$-module.

\begin{prop} \label{prop:xwedgey}
$\calO/R$ is projective of rank $3$ as an $R$-module, and there exists a unique quadratic map 
\[ \psi=\psi_\calO:\tbigwedge^2 (\calO/R) \to \tbigwedge^4\,\calO \]
with the property that
\begin{equation} \label{eqn:xwedgey}
\psi(x \wedge y) = 1 \wedge x \wedge y \wedge xy 
\end{equation}
for all $x,y \in \calO$.
\end{prop}

\begin{proof}
See Voight \cite[Lemma 2.1]{Voight:char}.
\end{proof}

The quadratic module $\psi_\calO \colon \tbigwedge^2 (\calO/R) \to \tbigwedge^4\,\calO$ in \Cref{prop:xwedgey} is called the \defi{canonical exterior form} of $\calO$.

\begin{prop} \label{prop:extcanfunct}
The association $\calO \mapsto \psi_\calO$ yields a functor from the category of 
\begin{center}
\emph{projective quaternion orders over $R$, under isomorphisms}
\end{center}
to the category of
\begin{center}
\emph{ternary quadratic modules, under similarity}.
\end{center}
\end{prop}

\begin{proof}
An isomorphism $\phi \colon \calO \to \calO'$ of quaternion $R$-orders induces a similarity
\[
\xymatrix{
\tbigwedge^2(\calO/R) \ar[r]^(.55){\psi} \ar[d]^{\wedge^2 \phi}_{\wr} & \tbigwedge^4\,\calO \ar[d]^{\wedge^4 \phi}_{\wr} \\
\tbigwedge^2(\calO'/R) \ar[r]^(.55){\psi'} & \tbigwedge^4\,\calO'
} \]
because for all $x,y \in \calO$ we have
\begin{equation} \label{eqn:wedge4makesitgo}
\begin{aligned}
(\wedge^4 \phi)(\psi(x \wedge y)) &= 1 \wedge \phi(x) \wedge \phi(y) \wedge \phi(xy) \\
&= 1 \wedge \phi(x) \wedge \phi(y) \wedge \phi(x)\phi(y) \\
&= \psi'(\phi(x) \wedge \phi(y)) = \psi'((\wedge^2 \phi)(x \wedge y))
\end{aligned}
\end{equation}
as desired.
\end{proof}

\begin{exm} \label{p:canonextfree}
Suppose that $\calO$ is free with a good basis $1,i,j,k$ \cite[22.4.7]{Voight:quatbook} and multiplication laws
\begin{equation} \label{eqn:Qeqns} 
\begin{aligned}
i^2 &= ui-bc & jk &= a\overline{i} \\ 
j^2 &= vj-ac &  ki &= b\overline{j} \\
k^2 &= wk-ab \quad & ij &= c\overline{k}.
\end{aligned}
\end{equation}
We now compute the canonical exterior form \cite[Example 2.5]{Voight:char}
\[ \psi=\psi_\calO:\tbigwedge^2(\calO/R) \to \tbigwedge^4\,\calO. \]
We choose bases, with
\[ \tbigwedge^2(\calO/R) \xrightarrow{\sim} R(j \wedge k) \oplus R(k \wedge i) \oplus R(i \wedge j) = R e_1 \oplus Re_2 \oplus Re_3 \]
and the generator $-1\wedge i \wedge j \wedge k$ for $\tbigwedge^4\,\calO$.  

With these identifications, the canonical exterior form $\psi \colon R^3 \to R$ has
\[ \psi(e_1)=\psi(j \wedge k)=1 \wedge j \wedge k \wedge jk = 1 \wedge j \wedge k \wedge (-ai)=a(-1\wedge i \wedge j \wedge k) \]
and 
\begin{align*} 
\psi(e_1+e_2)-\psi(e_1)-\psi(e_2) &= \psi(k\wedge (i-j)) - \psi(j \wedge k) - \psi(k \wedge i) \\
&= -1 \wedge k \wedge j \wedge ki - 1 \wedge k \wedge i \wedge kj \\
&= -w(1 \wedge k \wedge i \wedge j)=w(-1\wedge i \wedge j \wedge k).
\end{align*}
Continuing in this way, we see that 
\[ \psi(x(j \wedge k) + y(k \wedge i) + z(i \wedge j))=Q(xe_1+ye_2+ze_3)=Q(x,y,z) \]
with 
\begin{equation}  \label{eqn:Qabcuvwagain}
Q(x,y,z)=ax^2+by^2+cz^2+uyz+vxz+wxy, 
\end{equation}
so that 
\begin{equation}
\Clf^0(\psi_\calO) \simeq \calO.
\end{equation}

Therefore $\psi$ furnishes an inverse to the map in \Cref{thm:functorial} when $R$ is a PID.  In particular, we have obtained the \emph{same} quadratic form as constructed from the reduced norm \cite[Proposition 22.4.12]{Voight:quatbook}, so
\begin{equation} \label{eqn:nrdosharppsi}
N\nrd(\calO^\sharp) \sim \psi_{\calO}.
\end{equation}
% In particular, from Proposition \ref{prop:calOisisompr} we have for every nondegenerate ternary quadratic form $Q\colon R^3 \to R$ we have $\psi_{\Clf^0(Q)}$ is similar to $Q$.  (We reprove this for general $R$ in the next section.)  
\end{exm}

\begin{exm}
Suppose that $R$ is a Dedekind domain with field of fractions $F$.  Then there is a good pseudobasis for $\calO$
\begin{equation} \label{eqn:goodpseudo}
\calO = R \oplus \fraka i \oplus \frakb j \oplus \frakc k.
\end{equation}

The canonical exterior form of $\calO$, by the same argument as in \Cref{p:canonextfree} but keeping track of scalars, is given by
\[ \psi_\calO \colon \frakb\frakc e_1 \oplus \fraka\frakc e_2 \oplus \fraka\frakb e_3 \to \fraka\frakb\frakc \]
under the identification 
\begin{equation}
\begin{aligned}
\tbigwedge^4\,\calO \xrightarrow{\sim} \fraka\frakb\frakc \\
1 \wedge i \wedge j \wedge k \mapsto -1;
\end{aligned}
\end{equation} 
we again have 
\[ \psi_\calO(xe_1+ye_2+ze_3)=ax^2+by^2+cz^2+uyz+vxz+wxy \]
as in \eqref{eqn:Qabcuvwagain}, but now with $x,y,z$ are restricted to their respective coefficient ideals.  
% Repeating the same argument as in \ref{p:canonextfree}, we obtain again a similarity
%\begin{equation} \label{eqn:nrdpsiO}
%\nrd(\calO^\sharp) \sim \psi_{\calO} 
%\end{equation}
%as in \eqref{eqn:nrdosharppsi}.
\end{exm}

The even Clifford algebra and the canonical exterior form carry with them one global property (Steinitz class) that must be taken into account before we obtain an equivalence of categories.  Briefly, in addition to similarities one must also take into account twisted similarities, obtained not by a global map but by twisting by an invertible module.  

\begin{defn} \label{def:twistingdef}
A quadratic module $d\colon P \to I$ with $P,I$ projective of rank $1$ is called a \defi{twisting} quadratic module if the associated bilinear map $P \otimes P \to I$ is an $R$-module isomorphism.
\end{defn}

\begin{exm}
The quadratic module $d\colon R \to R$ by $z \mapsto z^2$ is a (trivial) twisting.

If $P$ is an invertible $R$-module, then the quadratic module
\begin{equation}
\begin{aligned}
P &\to P^{\otimes 2} \\
z &\mapsto z \otimes z
\end{aligned}
\end{equation}
is twisting.  
\end{exm}

\begin{defn}
Let $Q\colon M \to L$ be a quadratic module and let $d\colon P \to I$ be a twisting quadratic module.  The \defi{twist} of $Q$ by $d$ is the quadratic module
\begin{align*}
Q \otimes d \colon M \otimes P &\to L \otimes I \\
x \otimes z &\mapsto Q(x) \otimes d(z).
\end{align*}

A \defi{twisted similarity} between quadratic modules $Q\colon M \to L$ and $Q' \colon M' \to L'$ is tuple $(f,h,d)$ where $d\colon P \to I$ is a twisting quadratic module and $(f,h)$ is a similarity between $Q \otimes d$ and $Q'$.  
\end{defn}

\begin{exm} \label{exm:twistinvideal}
Let $Q\colon M \to L$ be a quadratic module and let $\fraka \subseteq R$ be an invertible fractional ideal of $R$.  Then $d\colon \fraka \to \fraka^2$ is twisting, and the twist of $Q$ by $\fraka$ can be identified with
\begin{align*}
Q \otimes d \colon \fraka M &\to \fraka^2 L \\
zx &\mapsto z^2 Q(x).
\end{align*}
If $\fraka=aR$ is principal, then $Q$ is similar to $Q \otimes d$ via the similarity $(f,h)$ obtained by scaling by $a$.  However, if $\fraka$ is not principal, then $Q$ may not be similar to $Q \otimes d$.
\end{exm}

\begin{lem}
Let $Q\colon M \to L$ be a quadratic module and let $d\colon P \to I$ be twisting.  Then there is a canonical isomorphism of $R$-algebras 
\[ \Clf^0(Q) \xrightarrow{\sim} \Clf^0(Q \otimes d). \]
\end{lem}

\begin{proof}
First, we have a canonical isomorphism
\begin{equation} \label{eqn:MMLP}
(M \otimes P) \otimes (M \otimes P) \otimes (L \otimes I)\spcheck \xrightarrow{\sim} M \otimes M \otimes L\spcheck
\end{equation}
coming from rearranging, the canonical map $d\colon P \otimes P \to I$ followed by the evaluation map $I \otimes I\spcheck \xrightarrow{\sim} R$.  Now recall the definition of the even Clifford algebra $\Clf^0(Q)$ \cite[22.2.1]{Voight:quatbook}: 
\[ \Clf^0(Q)=\Ten^0(M;L)/I^0(Q). \]
The canonical isomorphism \eqref{eqn:MMLP} induces an isomorphism $\Ten^0(M \otimes P;L \otimes I)$ that maps $I^0(Q \otimes d) \xrightarrow{\sim} I^0(Q)$, and the result follows.
\end{proof}

\begin{thm} \label{mthm:bigthmcliffisom}
Let $R$ be a noetherian domain.  Then the associations
\begin{equation} \label{eqn:mapsofsetscliffisom}
\begin{aligned}
\left\{ \begin{minipage}{28ex} 
\begin{center}
\textup{Nondegenerate ternary quadratic modules over $R$ \\ up to twisted similarity}
\end{center} 
\end{minipage}
\right\} &\leftrightarrow \left\{ \begin{minipage}{30ex} 
\begin{center}
\textup{Projective quaternion orders over $R$ up to isomorphism}
\end{center} 
\end{minipage}
\right\} \\
Q &\mapsto \Clf^0(Q) \\
\psi_\calO &\mapsfrom \calO
\end{aligned}
\end{equation}
are mutually inverse, discriminant-preserving bijections that are also functorial with respect to $R$.
\end{thm}

Before we begin the proof of the theorem, we need one preliminary lemma.  

\begin{lem} \label{lem:wedgewedge}
Let $M$ be a projective $R$-module of rank $3$.  Then there are canonical isomorphisms
\begin{align}
\tbigwedge^3 \bigl(\tbigwedge^2 M\bigr) &\xrightarrow{\sim} \bigl(\tbigwedge^3 M\bigr)^{\otimes 2} \label{eqn:wedgewedge1} \\
\tbigwedge^2 \bigl(\tbigwedge^2 M\bigr) &\xrightarrow{\sim} M \otimes \tbigwedge^3 M. \label{eqn:wedgewedge2}
\end{align}
\end{lem}

\begin{proof}
The proof is a bit of fun with multilinear algebra \cite[Lemma 3.7]{Voight:char}.  To illustrate, we give the proof in the special case where $M$ is completely decomposable $M=\fraka e_1 \oplus \frakb e_2 \oplus \frakc e_3$---so in particular the result holds when $R$ is a Dedekind domain (and, more generally using the splitting principle).  In this case, we have 
\[ \tbigwedge^2 M \simeq \fraka\frakb(e_1 \wedge e_2) \oplus \fraka\frakc(e_1 \wedge e_3) \oplus \frakb\frakc (e_2\wedge e_3) \]
and so
\[ \tbigwedge^3 \bigl(\tbigwedge^2 M\bigr) \simeq (\fraka\frakb\frakc)^2 (e_1 \wedge e_2) \wedge (e_1 \wedge e_3) \wedge (e_2 \wedge e_3) \]
agreeing with
\[ \bigl(\tbigwedge^3 M\bigr)^{\otimes 2} \simeq (\fraka\frakb\frakc)^2 (e_1 \wedge e_2 \wedge e_3)^{\otimes 2} \]
by rearranging the tensor in a canonical way.  The second isomorphism follows similarly.
\end{proof}

%\begin{rmk} \label{rmk:splittingmethod}
%There is a general method, called the \emph{splitting principle} (see e.g.\ Elman--Karpenko--Merkurjev \cite[Proposition 53.13]{EKM:quadforms} or Fulton \cite[\S 2]{Fulton:Schubert}), that allows one to reduce questions about modules (or vector bundles) to the case of a sum of invertible modules (line bundles).  For more on the connection to symmetric powers of wedge powers of modules and the relationship to problem of inner plethysm (related to Proposition \ref{prop:xwedgey}), see Weyman \cite[p.~63]{Weyman}.  For a $K$-theory version of the splitting principle, see Atiyah \cite[Corollary 2.7.11]{Atiyah:Ktheory}.
%\end{rmk}

\begin{proof}[Proof of \Cref{mthm:bigthmcliffisom}]
We have two functors, by \Cref{thm:functorial} and \Cref{prop:extcanfunct}, that are functorial with respect to the base ring $R$.  

We now compose them.  Let $Q\colon M \to L$ be a quadratic module with $\calO=\Clf^0(Q)$, and consider its canonical exterior form $\psi \colon \tbigwedge^2(\calO/R) \to \tbigwedge^4\,\calO$.  As $R$-modules, we have canonically 
\begin{equation} \label{eqn:canori}
\calO/R \simeq \tbigwedge^2 M \otimes L\spcheck. 
\end{equation}
By the isomorphism \eqref{eqn:wedgewedge2}, we have as the domain of $\psi$ the $R$-module 
\begin{equation} \label{eqn:domain1}
\begin{aligned} 
\tbigwedge^2 (\calO/R) \simeq \tbigwedge^2(\tbigwedge^2 M \otimes L\spcheck) \simeq \tbigwedge^2(\tbigwedge^2 M) \otimes (L\spcheck)^{\otimes 2} \\
\simeq M \otimes \tbigwedge^3 M \otimes (L\spcheck)^{\otimes 2}
\end{aligned}
\end{equation}
and as codomain we have by the canonical $R$-module isomorphism 
\[ \tbigwedge^3(\calO/R) \xrightarrow{\sim} \tbigwedge^4\,\calO \]
and the isomorphism \eqref{eqn:wedgewedge1}
\begin{equation} \label{eqn:codomain2}
\begin{aligned}
\tbigwedge^4\,\calO \simeq \tbigwedge^3 (\calO/R) &\simeq \tbigwedge^3 (\tbigwedge^2 M \otimes L\spcheck) \simeq \tbigwedge^3 (\tbigwedge^2 M) \otimes (L\spcheck)^{\otimes 3} \\
&\simeq
(\tbigwedge^3 M)^{\otimes 2} \otimes (L\spcheck)^{\otimes 3}.
\end{aligned}
\end{equation}

Now we twist.  Let 
\[ P \colonequals \tbigwedge^3 M \otimes (L\spcheck)^{\otimes 2} \] 
and let $d\colon P\spcheck \to (P\spcheck)^{\otimes 2}$ be the natural twisting quadratic module.  Then the twist $\psi \otimes d$ has domain and codomain canonically isomorphic to
\begin{equation} \label{eqn:tbigwedgealign}
\begin{aligned} 
\tbigwedge^2(\calO/R) \otimes P\spcheck &\simeq M \\
\tbigwedge^4\,\calO \otimes (P\spcheck)^{\otimes 2} &\simeq L 
\end{aligned}
\end{equation}
by \eqref{eqn:domain1}--\eqref{eqn:codomain2} so we have a quadratic form $\psi_{\Clf^0(Q)} \otimes d\colon M \to L$.  

We show that the composition $Q \mapsto \Clf^0(Q)=\calO \mapsto \psi_\calO$ is naturally isomorphic to the identity, via the twist $d$.  But to do this (and show the induced maps are similarities), since the above construction is canonical, we can base change to $F$ and check within the quadratic space $Q_F \colon V \to F$ where $V \colonequals M \otimes_R F$.  Choosing a basis, we find that the composition is the identity by \Cref{p:canonextfree}.

We may then conclude that the map of sets in \eqref{eqn:mapsofsetscliffisom} is a well-defined bijection: functoriality shows that the map is well-defined and that both maps are injective, and the composition shows that it is surjective.
\end{proof}

%\begin{exm}  \label{p:nrdscharpO}  
%One can also prove \Cref{mthm:bigthmcliffisom} by extending the definition in \ref{p:OB} using the reduced norm to a domain $R$ as follows: we define
%\begin{align*}
%\nrd^\sharp(\calO):(\calO^\sharp)^0 &\to \discrd(\calO)^{-1} \\
%\alpha &\mapsto \nrd(\alpha).
%\end{align*}
%The fact that $(\calO^\sharp)^0$ is projective of rank $3$, that $\discrd(\calO)$ is invertible as an $R$-module, and that the quadratic module takes values in $\discrd(\calO)$ follow locally, the latter from \eqref{eqn:Nxisharpis}.  Locally, this form is similar to the canonical exterior form \eqref{eqn:nrdosharppsi}, so this should come as no surprise.  
%\end{exm}

%\begin{rmk}
%When $2 \in R^\times$ (and more generally for certain schemes with $2$ invertible), Balmer--Calm\`es \cite{BalmerCalmes} develop a categorical theory of \emph{lax-similitude} that coincides with our notion of twisting.
%\end{rmk}

We now shift to isometries and upgrade the bijection to an equivalence of categories; this has \Cref{cor:bij} as an immediate corollary.

\begin{rmk}
Over a field, we used orientations \cite[\S 4.5]{Voight:quatbook} to get an equivalence of categories: the crux of the problem being that the isometry $-1:V \to V$ maps to the identity map on $\Clf^0 V$ under the Clifford functor.  However, $-1$ acts nontrivially on the \emph{odd} Clifford module $\Clf^1 V$, and the point of an orientation is to keep track of this action on piece of $\Clf^1 V$ coming from the center.  Restricting an orientation of $V$ would also work here; but we add a different and more functorial structure on quaternion orders to allow for these extra morphisms.
\end{rmk}

%Let $Q\colon M \to L$ be a nondegenerate quadratic module with $M$ of finite odd rank $n$.  When we restrict to morphisms as isometries, we require that the map act as the identity on the codomain $L$; to this end, we will have to somehow remember this codomain, and its compatibility with the other structures of the even Clifford algebra.  We make the following definition.

\begin{defn}
Let $N$ be an invertible $R$-module.  A \defi{parity factorization} of $N$ is a pair of invertible $R$-modules $P,L$ and an isomorphism $p:N \xrightarrow{\sim} P^{\otimes 2} \otimes L$.  If $N$ has a parity factorization, we call it \defi{paritized}.

An \defi{isomorphism}\index{parity factorization!isomorphism} $(N,p)$ to $(N',p')$ of paritized  invertible $R$-modules is a pair of isomorphisms $N \simeq N'$, $P \simeq P'$ such that the diagram
\begin{equation}
\begin{gathered}
\xymatrix{
N \ar[r]^(0.40){p} \ar[d]^{\wr} & P^{\otimes 2} \otimes L \ar[d]^{\wr} \\
N' \ar[r]^(0.37){p'} & P'^{\otimes 2} \otimes L
} 
\end{gathered}
\end{equation}
commutes.  
\end{defn}

We \emph{do} want to fix the module $L$ in order to work with isometries; this serves as an anchor for our construction.

\begin{defn}
Let $\calO$ be a quaternion $R$-order.  Then $\calO$ is \defi{paritized} if $\calO$ is equipped with a parity factorization of $\tbigwedge^4\,\calO$.  An isomorphism of paritized quaternion orders is an isomorphism $\phi \colon \calO \simeq \calO'$ and an isomorphism of parity factorizations with the isomorphism $\tbigwedge^4\,\calO \simeq \tbigwedge^4\,\calO'$ given by $\wedge^4 \phi$.  
\end{defn}

\begin{exm}
Every invertible $R$-module $N$ has the identity parity factorization, with $P=R$ and $L=N$; up to isomorphism, every other differs by a choice of isomorphism class of $P$.  So parity factorizations can be thought of as factorizations according to parity inside $\Pic R$.
\end{exm}

\begin{exm}
The main example of a parity factorization comes from the even Clifford construction.  Let $Q\colon M \to L$ be a nondegenerate ternary quadratic module.  Let $\calO = \Clf^0(Q)$, and let $P=\tbigwedge^3 M \otimes (L\spcheck)^{\otimes 2}$.  Then by \eqref{eqn:codomain2}, we have a canonical parity factorization
\begin{equation} \label{eqn:pQDef} 
p_Q\colon \tbigwedge^4 \calO \simeq (\tbigwedge^3 M)^{\otimes 2} \otimes (L\spcheck)^{\otimes 3} \simeq P^{\otimes 2} \otimes L. 
\end{equation}
\end{exm}

\begin{lem} \label{lem:autroh1}
Let $(\calO,p)$ be a paritized quaternion $R$-order.  Then 
\[ \Aut_R(\calO,p) \simeq \Aut_R(\calO) \times \{\pm 1\}. \]
\end{lem}

\begin{proof}
By definition, an automorphism of $(\calO,p)$ is a pair of automorphisms $(\phi,h)$ with $\phi \in \Aut_R(\calO)$ and $h \in \Aut_R(P)$ such that the diagram
\[ 
\xymatrix{
\tbigwedge^4 \calO \ar[r]^(0.40){p} \ar[d]_{\wedge^4 \phi}^{\wr} & P^{\otimes 2} \otimes L \ar[d]^{\wr} \\
\tbigwedge^4 \calO  \ar[r]^(0.40){p} & P^{\otimes 2} \otimes L
} \]
commutes.  The choice $\phi=\id_\calO$ and $h=-1$ gives us an automorphism we denote by $-1$.  Since $L$ is fixed, it plays no role in this part.

We claim that if $\phi \in \Aut_R(\calO)$ is an $R$-algebra isomorphism, then there is a unique $h \in \Aut_R(P)$ up to $-1$ such that $(\phi,h) \in \Aut_R(\calO,p)$.  We will prove this over each localization $R_{(\frakp)}$, including $R_{(0)}=F$: once we choose such an $h$ over $F$, it follows that $\pm h \in \Aut_{R_{(\frakp)}}(P_{(\frakp)})$ and hence by intersecting $h \in \Aut_R(P)$.  
So we may suppose $\calO$ is free with good basis and $P \simeq R$, and $\tbigwedge^4 \phi$ acts by $\det \phi$; we want to show that $\det \phi=h^2$ is a square.  This follows from \cite[22.3.15]{Voight:quatbook}: we have $\phi=\adj(\rho)$ where $\rho \in \GL_3(R)$ is the action on the ternary quadratic module $M \simeq R^3$ associated to $\calO$, and $\det \phi = \det \adj(\rho) = (\det \rho)^2$.  
\end{proof}

\begin{exm}
We recall the twist construction employed in \eqref{eqn:codomain2} that we now define in the above terms.  Let $(\calO,p)$ be a paritized quaternion $R$-order, with the parity factorization 
\[ p \colon \tbigwedge^4\,\calO \xrightarrow{\sim} P^{\otimes 2} \otimes L. \]
Let $\psi_{\calO,p}:\tbigwedge^2(\calO/R) \to \tbigwedge^4\,\calO$ be the canonical exterior form.  We then define the quadratic module 
\begin{equation} 
\psi_{\calO,p} \colonequals p \circ \psi_\calO \otimes P\spcheck\colon \tbigwedge^2(\calO/R) \otimes P\spcheck \to L
\end{equation}
where $p$ induces an isomorphism $\tbigwedge^4\,\calO \otimes (P\spcheck)^{\otimes 2} \simeq L$.
\end{exm}

We now have the following theorem.

\begin{thm} \label{thm:invassocfunc}
The associations
\begin{align*}
Q &\mapsto (\Clf^0(Q),p_{Q}) \\
\psi_{\calO,p} &\mapsfrom (\calO,p)
\end{align*}
are functorial and provide a discriminant-preserving equivalence of categories between
\begin{center}
\emph{nondegenerate ternary quadratic modules over $R$ \\ under isometries}
\end{center}
and
\begin{center}
\emph{paritized projective quaternion $R$-orders under isomorphisms}
\end{center}
that is functorial in $R$.
\end{thm}

\begin{proof}
First, we show the associations are functorial.  If $Q\colon M \to L$ and $Q' \colon M' \to L$ are isometric (nondegenerate ternary) quadratic modules under $f\colon M \to M'$, then by functoriality of the even Clifford algebra, this induces an isomorphism $f\colon \calO \simeq \calO'$ and thereby an isomorphism 
\[ P=\tbigwedge^3 M \otimes (L\spcheck)^{\otimes 2} \simeq P' = \tbigwedge^3 M' \otimes (L\spcheck)^{\otimes 2} \] 
and then of parity factorizations
\begin{equation}
\begin{gathered}
\xymatrix{
\tbigwedge^4\,\calO \ar[r]^(0.40){p_{Q}} \ar[d]^{\wr} & P^{\otimes 2} \otimes L \ar[d]^{\wr} \\
\tbigwedge^4\,\calO' \ar[r]^(0.37){p_{Q'}} & P'^{\otimes 2} \otimes L
}. 
\end{gathered}
\end{equation}
Conversely, if $(\calO,p)$ and $(\calO',p')$ are isomorphic paritized quaternion $R$-orders under $\phi \colon \calO \to \calO'$ and $P \simeq P'$, then we get an isometry 
\begin{equation}
\begin{gathered}
\xymatrix{
\tbigwedge^2 \calO/R \otimes P\spcheck \ar[r]^(0.7){\psi_{\calO,p}} \ar[d]^{\wr} & L \ar@{=}[d] \\
\tbigwedge^2 \calO'/R \otimes (P')\spcheck \ar[r]^(0.7){\psi_{\calO',p'}} & L
} 
\end{gathered}
\end{equation}
by \Cref{prop:extcanfunct} (see \eqref{eqn:wedge4makesitgo}): the similitude factor is the identity by construction.  

We now tackle the two compositions and show they are each naturally isomorphic to the identity.  Let $(\calO,p)$ be a paritized quaternion $R$-order.  We first associate $(\psi_{\calO,p},\id)$, and let $M=\tbigwedge^2 \calO/R \otimes P\spcheck$ be the domain of $\psi_{\calO,p}$.  We then associate its even Clifford algebra.  As $R$-modules, we have canonical isomorphisms
\begin{equation} \label{eq:firstclif}
\begin{aligned} 
\Clf^0(\psi_{\calO,p}) &= R \oplus \tbigwedge^2 M \otimes L\spcheck
=R \oplus \tbigwedge^2(\calO/R \otimes P\spcheck) \otimes L\spcheck \\
&\simeq R \oplus \tbigwedge^2(\calO/R) \otimes (P^{\otimes 2} \otimes L)\spcheck \\
&\simeq R \oplus \calO/R \otimes \tbigwedge^3(\calO/R) \otimes (P^{\otimes 2} \otimes L)\spcheck \\
&\simeq R \oplus \calO/R
\end{aligned}
\end{equation}
where we have used \eqref{eqn:wedgewedge2} and in the last step we used the parity factorization $p$ giving a natural isomorphism of the last piece to $R$.  To check that the corresponding map is an $R$-algebra homomorphism, by functoriality we can do so over $F$, and we suppose that $B$ is given by a good basis, and then the verification is as in \eqref{p:canonextfree}.  To finish, we show that the parity factorization is also canonically identified: we have
\begin{align*} 
\tbigwedge^3 M \otimes (L\spcheck)^{\otimes 2} &= \tbigwedge^3 (\tbigwedge^2(\calO/R) \otimes P\spcheck) \otimes (L\spcheck)^{\otimes 2} \\
&\simeq \tbigwedge^3 (\tbigwedge^3 \calO/R)^{\otimes 2} \otimes (P\spcheck)^{\otimes 3} \otimes (L\spcheck)^{\otimes 2} \simeq P
\end{align*}
where now we use \eqref{eqn:wedgewedge1} and then again (twice) the parity factorization.  Therefore we have a natural isomorphism of parity factorizations
\begin{equation}
\begin{gathered}
\xymatrix{
\tbigwedge^4 \Clf^0(\psi_{\calO,p}) \ar[r]^(0.45){p_{\psi_{\calO,p}}} \ar[d]^{\wr} & (\tbigwedge^3 M \otimes L\spcheck)^{\otimes 2} \otimes L \ar[d]^{\wr} \\
\tbigwedge^4\,\calO \ar[r]^{p} & P^{\otimes 2} \otimes L
} 
\end{gathered}
\end{equation}
This completes the verification that the composition in this order is naturally isomorphic to the identity.  

Now for the second composition.  Let $Q\colon M \to L$ be a (nondegenerate ternary) quadratic module over $R$.  We associate $\calO=\Clf^0 Q$ and $p_{Q}$ its parity factorization, and then $\psi_{\calO,p_Q}$.  In \eqref{eqn:tbigwedgealign}, using \eqref{eqn:domain1}--\eqref{eqn:codomain2}, we showed that we had an natural isometry
\begin{equation} \label{eqn:thebigisometry}
\begin{gathered}
\xymatrix{
\tbigwedge^2 \calO/R \otimes P\spcheck \ar[r]^(0.7){\psi_{\calO,p_Q}} \ar[d]_f^{\wr} & L \ar@{=}[d] \\
M \ar[r]^{Q} & L
} 
\end{gathered}
\end{equation}
and this completes the proof.  
\end{proof}

\end{document}